\documentclass{amsart}

\makeatletter
\let\@wraptoccontribs\wraptoccontribs
\makeatother

\usepackage{amsmath, amssymb, amsthm, graphicx, hyperref}

\usepackage{amsmath,amssymb,amsfonts}  
\usepackage[usenames,dvipsnames,svgnames]{xcolor}
\usepackage{tikz}

\usetikzlibrary{arrows,snakes,shapes.geometric,backgrounds} 
\usetikzlibrary[arrows,snakes,shapes.geometric] 

\newtheorem{thm}{Theorem}[section]
\newtheorem{prop}[thm]{Proposition}
\newtheorem{lem}[thm]{Lemma}
\newtheorem{cor}[thm]{Corollary}
\newtheorem*{main}{Main Theorem}
\newtheorem*{key}{Key Lemma}

\theoremstyle{definition}
\newtheorem{defn}[thm]{Definition}
\newtheorem{ex}[thm]{Example}
\newtheorem{rem}[thm]{Remark}

\newcommand{\Bm}[1]{\partial_{M}^{N#1}}

\newcommand{\Bmg}{\partial_{M}^{g(N)}}
\newcommand{\Bmprime}{\partial_M^{N'}}
\newcommand{\bm}{\partial_{M}}

\newcommand{\Ncal}{\mathcal{N}}

\newcommand{\Hyp}{\mathbb{H}}

\newcommand{\N}{\mathbb{N}}
\newcommand{\R}{\mathbb{R}}

\newcommand{\teich}{\mathcal{T}(S)}
\newcommand{\thick}{\mathcal{T}_{\epsilon}(S)}
\newcommand{\thickprime}{\mathcal{T}_{\epsilon'}(S)}

\newcommand{\Mod}{\mathrm{Mod}(S)}
\newcommand{\MC}{\mathcal{M}(S)}
\newcommand{\pmf}{\mathbb{P}\mathrm{MF}(S)}

\title{Morse Boundaries of Proper Geodesic Metric Spaces}

\author{Matthew Cordes}

\thanks{The author was partially supported by NSF grant DMS-1106726.}

\contrib[with a corrigendum by]{Matthew Cordes, Alessandro Sisto, Stefanie Zbinden}

\begin{document}

\begin{abstract}  We introduce a new type of boundary for proper geodesic spaces, called the Morse boundary, that is constructed with rays that identify the ``hyperbolic directions" in that space. This boundary is a quasi-isometry invariant and thus produces a well-defined boundary for any finitely generated group. In the case of a proper $\mathrm{CAT}(0)$ space this boundary is the contracting boundary of Charney and Sultan, and in the case of a proper Gromov hyperbolic space this boundary is the Gromov boundary. We prove three results about the Morse boundary of Teichm\"uller space. First, we show that the Morse boundary of the mapping class group of a surface is homeomorphic to the Morse boundary of the Teichm{\"u}ller space of that surface. Second, using a result of Leininger and Schleimer, we show that Morse boundaries of Teichm{\"u}ller space can contain spheres of arbitrarily high dimension. Finally, we show that there is an injective continuous map of the Morse boundary  of Teichm{\"u}ller space into the Thurston compactification of Teichm{\"u}ller space by projective measured foliations. 

	An appendix includes a corrigendum to the paper introducing refined Morse gauges to correct the proof and statement of Lemma 2.10. 
	\end{abstract}

\maketitle

\section{Introduction and Background}

Boundaries have been an extremely fruitful tool in the study of hyperbolic groups. One classical boundary, the visual boundary, is defined to be equivalence classes of geodesic rays, where one ray is equivalent to the other if they have bounded Hausdorff distance. Roughly, one topologizes the boundary by declaring open neighborhoods of a ray $\gamma$ to be the rays that stay close to $\gamma$ for a long time. Gromov in \cite{gromov:1987aa} showed that a quasi-isometry of a hyperbolic metric space induces a homeomorphism on the visual boundary, giving the notion of a well-defined boundary of a hyperbolic group. 

The visual boundary for a $\mathrm{CAT}(0)$ space can be similarly defined. Unfortunately, Croke and Kleiner show that the visual boundary of a $\mathrm{CAT}(0)$ space is not generally preserved under quasi-isometry \cite{Croke:2000aa}. Charney and Sultan in \cite{charney:2015aa} showed that if one restricts their attention to rays with hyperbolic-like behavior, so-called contracting rays, then one can construct a quasi-isometry invariant  boundary for any complete $\mathrm{CAT}(0)$ space. In this paper we show that if one considers another class of rays with hyperbolic-like behavior, Morse rays, then one can generalize the boundary of Charney and Sultan to construct a quasi-isometry invariant boundary for any proper geodesic space. We call this boundary the \emph{Morse boundary}. In the cases of proper $\mathrm{CAT}(0)$ spaces and hyperbolic spaces, the Morse boundary coincides with the contracting boundary of Charney and Sultan \cite{charney:2015aa} and the visual boundary respectively. 

The generality in which this boundary is defined means it is a new quasi-isometry invariant for every finitely generated group. While the Morse boundary of a group may be empty, Sisto in \cite{sisto:2016aa} showed that every group in the class of acylindrically hyperbolic groups unified by Osin in \cite{osin:2016aa} (which includes mapping class groups, $\mathrm{Out}(F_n)$, relatively hyperbolic groups, among others) will always have non-empty Morse boundary making the Morse boundary of particular interest for these groups.

A geodesic $\gamma$ in a metric space is called \emph{$N$-Morse}, where $N$ is a function 
$[1,\infty) \times [0, \infty) \to [0,\infty)$, if for any $(\lambda, \epsilon)$-quasi-geodesic $\sigma$ with endpoints on $\gamma$, we have $\sigma \subset \Ncal_{N(\lambda, \epsilon)}(\gamma)$, the $N(\lambda, \epsilon)$-neighborhood of $\gamma$. In a $\delta$-hyperbolic space, the well-known Morse lemma tells us that every ray is Morse and furthermore they are all $N$-Morse for the same $N$. In fact, if every geodesic in a space is $N$-Morse, then the space $\delta$-hyperbolic. On the other hand, no geodesic ray in Euclidean space is Morse. It is in this way that Morse geodesics point in the ``hyperbolic directions" of the space. 

The following key lemma states that if a ray is close to an $N$-Morse ray, then it is uniformly close where the uniform constant depends only on $N$. Variants of this lemma will be repeatedly useful.

\begin{key} Let $X$ be a geodesic metric space and $\alpha \colon [0, \infty) \rightarrow X$ be an $N$-Morse geodesic ray. Let $\beta \colon [0, \infty) \rightarrow X$ be a geodesic ray with $\alpha(0)=\beta(0)$ such that $\alpha, \beta$ have bounded Hausdorff distance. Then there exists a constant $\delta_N$ that depends only on $N$ such that $d(\alpha(t), \beta(t)) < \delta_N$ for all $t \in [0, \infty)$.
\end{key}

The Morse boundary of a space $X$, $\bm X$, is the set of all Morse geodesic rays in $X$ where two geodesic rays $\gamma, \gamma' \colon [0, \infty) \to X$ are identified if there exists a constant $K$ such that $d(\gamma(t), \gamma'(t))< K$ for all $t>0$. We denote an equivalence class of a ray $\alpha\in \bm X$ by $[\alpha]$.
If we fix a basepoint $p$ and a Morse function $N$ and consider the subset of the boundary that consists of all rays in $X$ with Morse function at most $N$: $$\Bm{} X_p= \{[\alpha] \mid \exists \beta \in [\alpha] \text{ that is an $N$-Morse geodesic ray with } \beta(0)=p\},$$ 
the key lemma affords us the ability to topologize this set in a similar manner as one does for the visual boundary of hyperbolic spaces. We endow the Morse boundary with the topology of the direct limit over all Morse gauges and show that this boundary is independent of basepoint. We also show that it is a visibility space, that is, any two distinct points in the Morse boundary can be joined by a Morse bi-infinite geodesic. In summary we prove:

\begin{main} Given a proper geodesic space $X$, the Morse boundary, $\bm X$, equipped with the direct limit topology, is 
\begin{enumerate}
\item a visibility space;
\item invariant under quasi-isometry of $X$; and
\item homeomorphic to the visual boundary if $X$ is hyperbolic and the contracting boundary if $X$ is proper $\mathrm{CAT}(0)$.
\end{enumerate}
\end{main}

In Section \ref{subsec:mpmaps} we identify a criterion which guarantees that a quasi-isometric embedding induces a topological embedding of the boundary. We first define a map between boundaries $\Omega \colon \partial_M X \to \partial_M Y$ to be \emph{Morse preserving} if for all Morse gauges $N$, $\partial_M^N X \hookrightarrow \partial_M^{N'} Y$ where $N'$ depends on $N$. We show that if a quasi-isometric embedding $f \colon X \to Y$ induces a Morse preserving map on the boundaries then the image of the Morse boundary of $X$ topologically embeds into the Morse boundary of $Y$.

We apply this result to Teichm\"uller space, $\teich$, with the Teichm\"uller metric. It was shown by Masur and Wolf in \cite{Masur:1995aa} that $\teich$ with this metric is not Gromov hyperbolic. But Minsky in  \cite{minsky:1996aa} showed that all geodesics in the $\epsilon$-thick part of Teichm\"uller space, $\thick$, are $N$-Morse where $N$ depends on $\epsilon$. So the $\bm \teich$ is non empty. Using a result of Leininger and Schleimer \cite{leininger:2014aa} we show that for any $n \geq 2$, there exists a surface of finite type $S$ and a Morse preserving map $\Omega \colon \bm \Hyp^n \to \bm \teich$. Thus for any $n \geq 2$, there exists a surface of finite type $S$ such that $\bm \teich$ contains a topologically embedded $S^{n-1}$.

We next show, using the Masur-Minsky subsurface projection machinery, that the Morse boundary of the mapping class group of a surface of finite type, $\Mod$, is homeomorphic to the Morse boundary of the Teichm\"uller space of that surface. Finally we show that there is a continuous injective map from the Morse boundary of  Teichm\"uller space to Thurston's compactification of $\teich$ by projective measured foliations. 

The author would like to thank his advisor Ruth Charney for her exceptional support and enthusiastic guidance. The author also like to thank Moon Duchin and Matthew Durham for the fruitful discussions and especially Moon for her invaluable feedback in the process of writing this paper. Finally the author with like to thank the anonymous referee for a very careful review and many helpful suggestions.

\subsection{Basic Definitions and Theorems} Let $(X,d)$ be a metric space. A \emph{geodesic} in $X$ is an isometric embedding $\gamma$ from a finite or infinite interval of $\R$ into $X$. We say that $(X,d)$ is a \emph{geodesic metric space} if every two points in $X$ are joined by a geodesic. We say $X$ is a \emph{proper metric space} if for every $x\in X$ and every $r >0$, the closed ball $\overline{B(x,r)}$ is compact.  

\begin{defn}[Hausdorff distance] Let $X$ be a metric space and let $\mathcal{N}_\eta(A)$ denote the $\eta$-neighborhood of a subset $A \subset X$. The Hausdorff distance between $A,B \subset X$ is defined by $$\inf\{\eta \mid A \subset \mathcal{N}_\eta(B), B \subset \mathcal{N}_\eta(A) \}.$$ 
\end{defn}

\begin{defn}[quasi-isometry; quasi-geodesic] A map $f \colon X \to Y$ between metric spaces is called a \emph{$(\lambda, \epsilon)$-quasi-isometric embedding}, where $\lambda \geq 1, \epsilon \geq 0$, if for every $a,b \in X$ $$\frac{1}{\lambda}d_X(a,b)- \epsilon \leq d_Y(f(a), f(b)) \leq \lambda d_X(a,b)+ \epsilon.$$ We say $f$ is a \emph{quasi-isometry} if there exists a constant $C \geq 0$ such that for every $y \in Y$ there exists an $a \in X$ such that $d_Y(y, f(a))<C$. If $X$ is a (possibly infinite) segment of $\R$, then we call the image of $f$ a \emph{$(\lambda, \epsilon)$-quasi-geodesic}.

Given a quasi-isometry $f \colon X \to Y$, there exists a \emph{quasi-inverse} $g\colon Y \to X$, which is itself a quasi-isometry such that there exists a constant $C$, depending only on $\lambda, \epsilon$, with the property that for all $x \in X, y \in Y$, $$d_X(x, (g \circ f)(x)) \leq C \text{ and } d_Y(y, (f \circ g)(y)) \leq C.$$
 \end{defn}

\begin{defn}[Morse (quasi)-geodesics] A (quasi)-geodesic $\gamma$ in a metric space is called \emph{$N$-Morse}, where $N$ is a function 
$[1,\infty) \times [0, \infty) \to [0,\infty)$, if for any $(\lambda, \epsilon)$-quasi-geodesic $\sigma$ with endpoints on $\gamma$, we have $\sigma \subset \Ncal_{N(\lambda, \epsilon)}(\gamma)$.
We call the function $N(\lambda, \epsilon)$ a \emph{Morse gauge}.\end{defn}

We will also use two easy corollaries of Arzel\`a-Ascoli \cite[Theorem 47.1]{munkres:2000aa}:

\begin{cor} Let $X$ be a proper metric space and $p \in X$. Then any sequence of geodesics $\beta_n \colon [0, L_n]  \to X$ with $\beta_n(0)=p$ and $L_n \to \infty$ has a subsequence that converges uniformly on compact sets to a geodesic $\beta \colon [0, \infty) \to X$. \end{cor}

\begin{cor} \label{cor:AA2}Let $X$ be a proper metric space. Let $\beta_n \colon [L_n, M_n] \to X$ be any sequence of geodesics such that $L_n \to -\infty, M_n \to \infty$, and every $\beta_n(0)$ is in a set of bounded diameter. Then the sequence $(\beta_n)$ has a subsequence that converges uniformly on compact sets to a geodesic $\beta \colon (-\infty, \infty) \to X$. \end{cor}


\begin{rem} I will repeatedly use Lemma 2.5 in \cite{charney:2015aa}. The statement of the lemma assumes a $\mathrm{CAT}(0)$ space, but the proof only requires a geodesic space.
\end{rem}

\section{Properties of Morse Geodesics}

The following lemma verifies that a quasi-geodesic with endpoints on a Morse geodesic segment has bounded Hausdorff distance with the geodesic.

\begin{lem} \label{near morse means hausdorff close} Let $X$ be a geodesic space and let $\alpha \colon [a, b] \to X$ be a $N$-Morse geodesic segment and let $\beta \colon [a', b'] \to X$ be a $(\lambda,\epsilon)$-quasi-geodesic such that $\alpha(a)=\beta(a')$ and $\alpha(b)=\beta(b')$. Then the Hausdorff distance between $\alpha$ and $\beta$ is bounded by $2N(\lambda,\epsilon')+(\lambda+\epsilon)$ where $\epsilon'= 2(\lambda+\epsilon)$ or if $\beta$ is continuous $2N(\lambda, \epsilon)$
\end{lem}

\begin{proof}

First assume that $\beta$ is continuous.
By definition $\beta \subset \mathcal{N}_{N(\lambda,\epsilon)}(\alpha)$

We now show that $\alpha \subset \mathcal{N}_{2N(\lambda, \epsilon)} (\beta)$. We follow an argument similar to Lemma 2.5 (3) \cite{charney:2015aa}. If $\alpha\subset \mathcal{N}_{N(\lambda, \epsilon)}(\beta)$ we have our bound. If not consider a maximal segment $[t, t'] \subset [a,b]$ such that $\alpha([t,t'])$ is disjoint from $\mathcal{N}_{N(\lambda, \epsilon)}(\beta)$. We know by continuity of $\beta$ that there exists a $z \in [a',b']$ such that $\beta(z)$ lies within $N(\lambda,\epsilon)$ of two points $\alpha(r), \alpha(r')$, with $r \in [a, t], r' \in [t', b]$. Thus by the triangle inequality, $d(\alpha(r), \alpha(r')) < 2N(\lambda,\epsilon)$ and since $\alpha$ is a geodesic any point on $\alpha$ between $\alpha(r)$ and $\alpha(r')$ is at most $N(\lambda,\epsilon)$ from one of $\alpha(r)$ and $\alpha(r')$ and thus any point in $[t,t']$ is within $2N(\lambda,\epsilon)$ of $z$. We conclude $\alpha \subset \mathcal{N}_{2N(\lambda,\epsilon)}(\beta)$.

If $\beta$ is not continuous, we use  Lemma 1.11 in \cite{bridson:1999fj} III.H to replace $\beta$ with $\beta'$, a continuous $(\lambda, \epsilon')$-quasi-geodesic such that $\beta'(a')=\beta(a')$ and $\beta'(b')=\beta(b')$ and the Hausdorff distance between $\beta$ and $\beta'$ is less than $(\lambda+ \epsilon)$. We do the proceeding proof with using $\beta'$ and allow for the Hausdorff distance between $\beta$ and $\beta'$ for the general estimate.
\end{proof}

We now show that triangles in a geodesic metric space with two $N$-Morse edges are slim.

\begin{lem} \label{lem:triangle slim} Let $X$ be a geodesic space and let $\alpha_1 \colon [0, A] \to X$ and $\alpha_2 \colon [0,B] \to Y$ be $N$-Morse geodesics such that $\alpha_1(0)=\alpha_2(0)=p$. Let $\gamma \colon [0,C] \to X$ be a geodesic joining $\alpha_1(A)$ and $\alpha_2(B)$. Then the triangle $\alpha_1 \cup \gamma \cup \alpha_2$ is $4N(3,0)$-slim.
\end{lem}

\begin{proof}Choose $x$ so that it is the nearest point on $\gamma$ to $p$ and let $\eta$ be a geodesic connecting $x$ and $p$. We first show that the concatenation $\phi_1=\gamma([0, x]) \cup \eta$ and $\phi_2=\bar{\eta} \cup \gamma([x,C])$ are $(3,0)$-quasi-geodesics. 

We show that $\phi_1$ is a (3,0)-quasi-geodesic. Since all of the segments are geodesics, we only need check the inequality for points $u$ and $v$ on different segments. Let $u \in \gamma([0, x])$ and $v \in \eta$. 
We know $d(v, \gamma(x)) \leq d(u,v)$ because $\gamma(x)$ is a nearest point to $p$. We also note that $d(u, \gamma(x)) \leq d(u,v) + d(v, \gamma(x))$ by the triangle inequality.  Let $d_{\phi_1}(u,v)$ denote the distance along $\phi_1$ between $u$ and $v$.
\begin{align*}
d(u,v) \leq d_{\phi_1}(u,v) =& d(u, \gamma(x)) + d(\gamma(x), v) \\
				\leq&  (d(u,v)+d(v, \gamma(x)) + d(u,v) \\
				\leq & 3d(u,v).
\end{align*} Thus we have our inequality.

The inequality for $\phi_2$ follows identically. Since the $\alpha_i$ are $N$-Morse and $\phi_i$ is a $(3,0)$-quasi-geodesics with endpoints on $\alpha_i$, it follows that $\gamma \subset \mathcal{N}_{N(3,0)}(\alpha_1 \cup \alpha_2)$. 

Using Lemma \ref{near morse means hausdorff close}, we know that the Hausdorff distance between $\phi_i$ and $\alpha_i$ is less than $2N(3,0)$ for $i=1,2$. Thus for every $t \in [0,A]$ either $d(\alpha_1(t), \gamma)< 4N(3,0)$ or $d(\alpha_1(t), \alpha_2)<4N(3,0)$. So $\alpha_1 \subset \mathcal{N}_{4N(3,0)}(\alpha_2 \cup \gamma)$. The final containment follows identically.
\end{proof}

Armed with the knowledge that triangles with two Morse edges are slim, we show that the third edge is also Morse. 

\begin{lem} \label{lem:third edge N'} Let $X$ be a geodesic space and let $\alpha_1 \colon [0, A] \to X$ and $\alpha_2 \colon [0,B] \to Y$ be $N$-Morse geodesics such that $\alpha_1(0)=\alpha_2(0)=p$. Let $\gamma \colon [0,C] \to X$ be a geodesic joining $\alpha_1(A)$ and $\alpha_2(B)$. Then $\gamma$ is $N'$-Morse for some $N'$ depending on $N$. \end{lem}

\begin{proof}By Lemma \ref{lem:triangle slim}, we know that the triangle $\alpha_1 \cup \gamma \cup \alpha_2$ is $4N(3,0)$-slim. It follows from the continuity of the distance function that there exists an $x\in [0,C]$ and $s_i\in [0, \infty)$ such that $d(\gamma(x), \alpha_i(s_i))<4N(3,0)$ for $i=1,2$. We define $\gamma_1$ to be the concatenation of $\gamma[0,x]$ and a geodesic between $\gamma(x)$ and $\alpha_1(s_1)$. Define $\gamma_2$ similarly with $\alpha_2$. We note that these are $(1,4N(3,0))$-quasi-geodesics with endpoints on $N$-Morse geodesics. By Lemma \ref{near morse means hausdorff close} we know that the Hausdorff distance between $\gamma_1$ and $\alpha_1[s_1,A]$ is bounded by $2N(1,4N(3,0))$ and thus the Hausdorff distance between $\gamma[0,x]$ and $\alpha_1[s_1,A]$ is bounded by $2N(1,4N(3,0))+4N(3,0)$. We get identical bounds for $\gamma[x,C]$ and $\alpha_2[s_2,B]$. By Lemma 2.5 (1) of \cite{charney:2015aa} they are $N''$-Morse where $N''$ depends on $N$.

Let $\sigma \colon [a,b] \to X$ be a $(\lambda, \epsilon)$ quasi-geodesic with endpoints on $\gamma$.  Let $z \in [a,b]$ be such that $\sigma(z)$ is a point on $\sigma$ closest to $\gamma(x)$. Define $\sigma_1$ to be the concatenation of $\sigma[a,z]$ and $[\sigma(z), \gamma(x)]$, a geodesic between $\gamma(x)$ and $\sigma(z)$. Define $\sigma_2$ similarly. We claim that $\sigma _1$ is a $(2\lambda+1,\epsilon)$-quasi-geodesic


It is enough to check the inequality for points $u$ and $v$ on different segments. Let $u=\sigma(t)$ for some $t \in [a,z]$ and $v \in [\sigma(z), \gamma(x)]$. We know that $d(v, \sigma(z))\leq d(u,v)$ because $\sigma(z)$ is a nearest point to $\gamma(x)$. We also note that $d(u, \sigma(z)) \leq d(u,v) + d(\sigma(z),v)$ by the triangle inequality.  Finally, note that $\sigma_1$ naturally parametrized and the difference in the parameters for $u$ and $v$ is $|z-t|+ d(\sigma(z), v)$. Putting this information together we get the two following inequalities:


\begin{align*}
d(u,v) &\leq d(\sigma(t), \sigma(z))+ d(\sigma(z),v) \\
	&\leq  \lambda |z-t| + \epsilon +d(\sigma(z),v) \\
	&\leq \lambda (|z-t| +d(\sigma(z),v)) +\epsilon \\
	&\text{and} \\
|z-t|+ d(\sigma(z), v) &\leq \lambda d(u, \sigma(z)) + \epsilon + d(\sigma(z), v) \\
	&\leq \lambda(d(u,v) + d(\sigma(z),v))+\epsilon +d(\sigma(z), v) \\
	&\leq \lambda(d(u,v)+ d(u,v)) + \epsilon + d(u,v) \\
	&\leq (2\lambda +1)d(u,v)+ \epsilon.				
\end{align*} Thus we have our inequalities.

The inequality for $\sigma_2$ follows identically. Thus since $\gamma[0,x]$ and $\gamma[x, C]$ are $N''$-Morse, then $\sigma \subset \mathcal{N}_{N''(2\lambda+1, \epsilon)}(\gamma)$. Thus $\gamma$ is $N'$-Morse for some $N'$ depending on $N$.
\end{proof}

The following proposition and its corollaries are the key lemmas used to construct the Morse boundary. We show that if a geodesic is bounded distance from a $N$-Morse geodesic for a long enough time, then they are close where the bound depends only on $N$ and the distance between their basepoints. Thus, geodesics with the same basepoint are actually uniformly close.

\begin{prop} \label{morse close} Let $X$ be a geodesic metric space. Let $\alpha \colon [0, \infty) \rightarrow X$ be an $N$-Morse geodesic ray. Let $\beta \colon [0, \infty) \rightarrow X$ be a geodesic ray such that $d(\alpha(t), \beta(t))<K$ for $t \in [A, A+D]$ for some $A \in [0, \infty)$ and $D \geq 6K$. Then for all $t \in [A+2K, A+D-2K]$, $d(\alpha(t), \beta(t))< 4N(1,2N(5,0)) + 2N(5,0)+ d(\alpha(0),\beta(0))$.
\end{prop}

\begin{proof} Let $A \geq 0$ and $D \geq 6K$ and $A'=A+D$. Choose $x$ so that $\beta(x)$ is a point nearest to $\alpha(A)$ on $\beta$ and similarly choose an $x'$ so that $\beta(x')$ is a nearest point to $\alpha(A')$. Note by the triangle inequality that $x \in [\max \{0, A-2K\}, A+2K]$ and $x' \in [A'-2K, A'+2K]$.

Choose a geodesic $\mu$ from $\alpha(A)$ to $ \beta(x)$ and $\nu$ from $\beta(x')$ to  $\alpha(A')$. We claim that the concatenation of geodesics: \begin{equation*} \phi=\mu \cup \beta([x,x']) \cup \nu \end{equation*} is a $(5,0)$-quasi geodesic. See Figure \ref{key lemma picture}.

Let $u, v \in \phi$. Since all of the segments are geodesics, we only need check when $u$ and $v$ lie on different segments. There are three cases: 

\emph{CASE 1:} $u \in \mu$ and $ v \in  \beta([x,x'])$. Let $\psi = \mu \cup  \beta([x,x'])$. We know that $d(u, \beta(x)) \leq d(u,v)$ because $\beta(x)$ is a nearest point to $\alpha(A)$. We also note that $d(\beta(x), v) \leq d(u,v) + d(u, \beta(x))$ by the triangle inequality.  Let $d_\psi(u,v)$ denote the distance along $\psi$ between $u$ and $v$.
\begin{align*}
d(u,v) \leq d_\psi(u,v) =& d(u, \beta(x)) + d(\beta(x), v) \\
				\leq& d(u,v) + (d(u,v)+d(u, \beta(x)) \\
				\leq& 3d(u,v).
\end{align*} Thus we have our inequality for these two segments.

\emph{CASE 2:} $ u \in \beta([x,x']) $ and $v \in \nu$. This case follows similarly to CASE 1.

\emph{CASE 3:} $ u \in \mu$ and $ v \in \nu$. First note that 
\begin{align*}
2K + d(u,v) \geq& d(\alpha(A), u) + d(u,v) + d(v, \alpha(A')) \\
		\geq& d(\alpha(A), \alpha(A'))=D.
\end{align*}
Thus $d(u,v) \geq D-2K$. Since $D \geq 6K$, $D/3 \geq 2K$. Thus $d(u,v) \geq D-D/3$ or $D \leq \frac{3}{2} d(u,v) $.

We also note that $d(\alpha(A), \beta(x)) < K \leq D/6$ and similarly $d(\beta(x'), \beta(A'))<K \leq D/6$. 

Putting the inequalities together, we see that $d(\beta(x), \beta(x')) \leq D+2K < 2D$. Let $\zeta$ be the geodesic from $u$ to $\beta(x)$ following $\mu$ and let  $\eta$ be the geodesic from $\beta(x')$ following $\nu$. Consider $\xi=\zeta \cup \beta[x,x'] \cup \eta$.
\begin{align*}
d(u,v) \leq d_{\xi}(u,v) =& d(u, \beta(x)) + d(\beta(x), \beta(x')) + d(\beta(x'), v) \\
				<& D/6 + 2D + D/6 \\
				<& 3D \\
				\leq& \frac{9}{2}d(u,v)				 
\end{align*} Thus we have our inequality for these segments.
Therefore we have that $\phi$ is a $(5,0)$-quasi geodesic.

If $K\geq N(5,0)$ then let $\alpha(y)$ and $\alpha(y')$ be the points closest to $\beta(x)$ and $\beta(x')$ respectively. Otherwise let $y=A$ and $y'=A'$. First note that both $d(\beta(x), \alpha(y))<N(5,0)$ and $d(\beta(x'), \alpha(y')) < N(5,0)$. We first claim that $y < A+2K$. If $y=A$ then were done. Else, by the triangle inequality, we get that $y<A+K+N(5,0)<A+2K$. Similarly, we can conclude that $y'>A'-2K$.

Since the path $[\alpha(y),\beta(x)]\cup\beta[x,x']\cup[\beta(x'), \alpha(y')]$ is a $(1, 2N(5,0))$-quasi-geodesic with endpoints on $\alpha$, by Lemma \ref{near morse means hausdorff close} we can conclude the Hausdorff distance between this path and $\alpha([y, y'])$ is bounded by $2N(1,2N(5,0))$. Therefore, it follows that the Hausdorff distance between $\alpha([y,y'])$ and $\beta([x,x'])$ is bounded by $2N(1,2N(5,0)) + N(5,0)$.

To see the parameterized distance bound, we follow the proof of  Proposition 10.1.4 in \cite{papadopoulos:2005qy} and conclude that for all $t \in [A+2K, A+D-2K]$, $d(\alpha(t), \beta(t))< 4N(1,2N(5,0)) + 2N(5,0)+ d(\alpha(0),\beta(0))$.
 \end{proof}

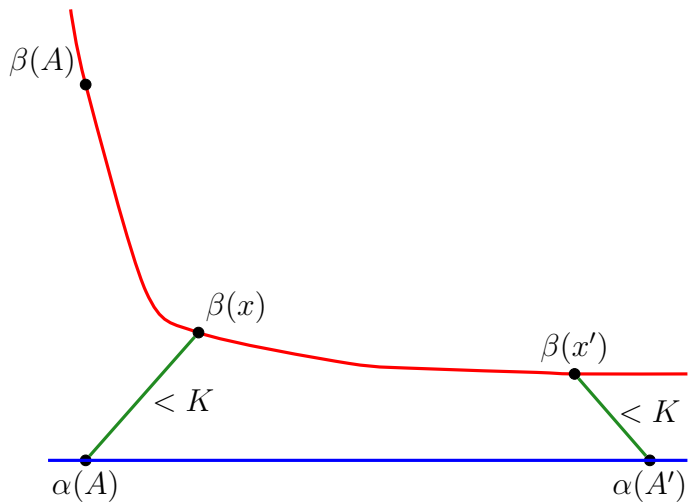
\begin{figure}
\begin{center}
\begin{tikzpicture}[inner sep= .5 mm] [font=\large]

\draw (0,5)coordinate(x) {};
\draw (0,2) coordinate (r) {};
\draw (0,0) coordinate (pigx) {};
\draw (2,3) coordinate (s) {};
\draw (7,4) coordinate (t) {};
\draw (8,5) coordinate (y) {};
\draw (8,3.5) coordinate (u) {};
\draw (8,0) coordinate (pigy) {};

\draw (.8,2.25) coordinate (int1) {};
\draw (1.5,1.7) coordinate (int2) {};

\draw (3.5,1.29) coordinate (int3) {};
\draw (4.5,1.22) coordinate (int4) {};
\draw (6,1.17) coordinate (int5) {};
\draw (6.5, 1.15) coordinate (int6) {};
\draw (8,1.15) coordinate (int7) {};
\draw (-.2, 6) coordinate (bstart) {};

\draw [very thick,color=red] plot[smooth=80] coordinates {(bstart) (x) (int1) (int2)   (int3)  (int4) (int5)  (int6)  (int7) };

\draw (0,5) node [circle,draw,fill=black,label=160:$\beta(A)$] (x) {};
\draw (0,0) node [circle,draw,fill=black,label=below:$\alpha(A)$] (aN) {};
\draw (1.5,1.7) node [circle, draw, fill=black, label=60:$\beta(x)$] (int2) {};
\draw (6.5, 1.15) node [circle, draw, fill=black, label=above:$\beta(x')$] (int6) {};
\draw(-.5, 0) coordinate (astart) {};
\draw(7.5, 0) node [circle, draw, fill=black, label=below:$\alpha(A')$] (aD) {};
\draw (8,0) coordinate  (a) {};

\draw [very thick, color=blue] (astart) -- (a);
\draw[very thick, color=ForestGreen] (aN) -- (int2);
\draw[very thick, color=ForestGreen] (aD) -- (int6);

\draw (1.3,.8) node (k1) {$<K$};
\draw (7.5,.65) node (k2) {$<K$};

\end{tikzpicture}
\end{center}
\caption{Picture of situation in Proposition \ref{morse close}}  \label{key lemma picture}
\end{figure}

{\bf Notation:} For $N=N(K,L)$ a Morse gauge, set $\delta_N=\max\{4N(1,2N(5,0)) + 2N(5,0), 8N(3,0)\}$. 

\begin{cor} \label{morse close same basepoint} Let $X$ be a geodesic metric space and $\alpha \colon [0, \infty) \rightarrow X$ be an $N$-Morse geodesic ray. Let $\beta \colon [0, \infty) \rightarrow X$ be a ray such that $\beta(0)= \alpha(0)$ and $d(\alpha(t), \beta(t))<K$ for $t \in [0, D]$ for some $D \geq 6K$. Then  $d(\alpha(t), \beta(t))<\delta_N$ for all $t \in [0,D-2K]$.
\end{cor}

\begin{proof} Follow the proof of Proposition \ref{morse close} in this case.
\end{proof}

\begin{cor} \label{asymp morse eventually close} 
Let $X$ be a geodesic metric space and $\alpha \colon [0, \infty) \rightarrow X$ be an $N$-Morse geodesic ray. Let $\beta \colon [0, \infty) \rightarrow X$ be a ray such that $d(\alpha(t), \beta(t))<K$ for all $t \in [0, \infty)$ (i.e., $\beta \in [\alpha]$). Then for all $t \in [2K, \infty)$, $ d(\alpha(t), \beta(t)) <\delta_N +d(\alpha(0), \beta(0))  $. In particular if $\alpha(0)=\beta(0)$, then $d(\alpha(t), \beta(t)) < \delta_N$ for all $t \in [0, \infty)$.
\end{cor}

\begin{proof} The first statement follows from Proposition \ref{morse close} as $A \in [0, \infty)$ is arbitrary. The second follows from Corollary \ref{morse close same basepoint}.
\end{proof}

The next result is similar in flavor to the preceding results and we will use it in showing that the Morse boundary is a quasi-isometry invariant.

\begin{lem} \label{lem:twodelta} Let $X$ be a geodesic space and let $\alpha_1, \alpha_2 \colon [0,A] \to X$ be $N$-Morse geodesics with $\alpha_1(0)=\alpha_2(0)$. If $d(\alpha_1(s), \mathrm{im}(\alpha_2))< K$, for some $K>0$ and $s \in [0,A]$, then $d(\alpha_1(t), \alpha_2(t))\leq 8N(3,0) < \delta_N$ for $t < s-K-4N(3,0)$. \end{lem}
\begin{proof} Follow exactly the proof of Lemma 1.15 in  \cite{bridson:1999fj} III.H using the slimness of triangles with two $N$-Morse legs shown in Lemma \ref{lem:triangle slim}. \end{proof}

The next lemma states that given an $N$-Morse geodesic $\gamma$ with basepoint $p$, we can construct a geodesic with basepoint $p'$ asymptotic to $\gamma$ that is $N'$-Morse where $N'$ depends only on $N$ and $d(p,p')$. This is important in showing basepoint independence of the Morse boundary.

\begin{lem} \label{basepoint change} Let $X$ be a proper geodesic metric space, $p,p'\in X$ and $\alpha \colon [0, \infty) \to X$ an $N$-Morse geodesic ray such that $\alpha(0)=p$. Then there exists an $N'$- Morse geodesic ray $\beta \colon [0, \infty) \to X$ asymptotic to $\alpha$ with $\beta(0)=p'$, $N \leq N'$ (where $N'$ depends only on $d(p,p')$ and $N$), and $d(\alpha(t), \beta(t))< 4N(1,2d(p,p')) + 3d(p,p')$ for all $t \in [0, \infty)$. \end{lem}

\begin{proof} Let $\{\beta_n\}_{n \in \N}$ be a sequence of geodesics joining $p'$ and $\alpha(n)$. Let $\bar{\beta_0}$ be $\beta_0$ with opposite parameterization. We note that the concatenation $\phi_n=\bar{\beta_0} \cup \beta_n$ is a $(1, 2d(p,p'))$-quasi-geodesic with endpoints on $\alpha$ and thus by Lemma \ref{near morse means hausdorff close} we know $\phi_n$ is  Hausdorff distance at most $2N(1, 2d(p,p'))$ away from $\alpha$. By Arzel\`a-Ascoli a subsequence $\{ \beta_{n(i)} \}$ converges uniformly on compact sets to a ray $\beta$. Since all of the $\beta_n$ are at Hausdorff distance at most $2N(1, 2d(p,p'))+d(p,p')$ away from $\alpha$, the Hausdorff distance between $\alpha$ and $\beta$ is identically bounded.  It now follows from Lemma 2.5 (1) in \cite{charney:2015aa} that $\beta$ is an $N'$-Morse ray where $N'$ depends only on $N$ and $d(p,p')$. To see the parameterized distance bound, we use Proposition 10.1.4 in \cite{papadopoulos:2005qy}. 
\end{proof}

To show the Morse boundary is a quasi-isometry invariant, we show that under a quasi-isometry $N$-Morse geodesic rays are sent to quasi-geodesics rays near $N'$-Morse geodesic rays where $N'$ depends only on the quasi-isometry constants and $N$.  We will use this lemma to show that quasi-isometries induce maps on the Morse boundary.

\begin{lem} \label{straighten quasi-geodesics} Let $X$ and $Y$ be proper geodesic spaces and let $f \colon X \to Y$ be a $(\lambda, \epsilon)$-quasi-isometry. Then for any $N$-Morse geodesic ray $\gamma$ based at $p$, $f \circ \gamma$ stays bounded distance from an $N'$-Morse geodesic ray $\beta$ based at $f(p)$ where $N'$ depends only on $\lambda, \epsilon$, and $N$. \end{lem}

\begin{proof} We follow closely the proof of Corollary 2.10 in \cite{charney:2015aa}. By Lemma 2.5 (2) in \cite{charney:2015aa} we know that $f \circ \gamma$ is an $N''$-Morse quasi-geodesic with $N''$ only depending on $\lambda, \epsilon$ and $N$. Let $\beta_n$ be a geodesic segment from $f(\gamma(0))$ to $f(\gamma(n))$ for all $n \in \N$. Since $f \circ \gamma$ is a $(\lambda, \epsilon)$-quasi-geodesic, then by Lemma 2.5 (3) in \cite{charney:2015aa} there exists a constant $C$ that depends only on $\lambda, \epsilon$ and $N''$ so that every $\beta_n$ is within Hausdorff distance $C$ of $f \circ \gamma |_{[0,n]}$. Thus by Arzel\`a-Ascoli there exists a subsequence $\beta_{n(i)}$ that converges to a geodesic ray $\beta$ that is at most Hausdorff distance $C$ from $f \circ \gamma$. We use Lemma 2.5 (1) in \cite{charney:2015aa} to conclude that $\beta$ is $N'$-Morse where $N'$ depends only on $N''$, $C$. 
\end{proof}

\begin{lem} \label{M-Morse uniformly converge to M-Morse} Let X be a geodesic space and let $\{\gamma_i \colon [0, \infty) \to X \}$ be a sequence of $N$-Morse geodesic rays that converge uniformly on compact sets to a geodesic ray $\gamma$. Then $\gamma$ is $N$-Morse.
\end{lem}

\begin{proof} Let $\eta > 0$. Let $\beta$ be a $(\lambda, \epsilon)$-quasi geodesic with end points on $\gamma$. Since the $\gamma_i$ converge uniformly on compact sets and are $N$- Morse there exists an $I \in \N$ such that for any $i \geq I$, $d(\gamma_i(t), \gamma(t))< \eta$ on the interval between the endpoints of $\beta$. Thus $\gamma$ is an $N'$-Morse geodesic where $N' < N+\eta$. Since this is true for all $\eta >0$, $\gamma$ is an $N$-Morse geodesic. 
\end{proof}

\section{The Morse Boundary} \label{sec:morse boundary defn}

As a set, the Morse boundary of $X$ with basepoint $p$, $\bm X_p$, is the collection all Morse geodesic rays in $X$ with basepoint $p$ where two geodesic rays $\gamma, \gamma' \colon [0, \infty) \to X$ are identified if there exists a constant $K$ such that $d(\gamma(t), \gamma'(t))< K$ for all $t>0$. We denote an equivalence class of a ray $\alpha\in \bm X$ by $[\alpha]$.

In order to topologize the entire boundary, we first topologize pieces of the boundary and take a direct limit. Consider the subset of the Morse boundary  \begin{equation*} \Bm{} X_p= \{[\alpha] \mid \exists \beta \in [\alpha] \text{ that is an $N$-Morse geodesic ray with } \beta(0)=p\}. \end{equation*}

We toplogize $\Bm{} X_p$ following \cite{bridson:1999fj} III.H. Let $X$ be a proper geodesic space. Fix a basepoint $p \in X$. We define convergence in $\Bm{} X_p$ by: $x_n \to x$ as $n \to \infty$ if and only if there exists $N$-Morse geodesic rays $\alpha_n$ with $\alpha_n(0)=p$ and $[\alpha_n]= x_n$ such that every subsequence of $\{\alpha_n\}$ contains a subsequence that converges uniformly on compact sets to a geodesic ray $\alpha$ with $[\alpha]=x$. By Lemma \ref{M-Morse uniformly converge to M-Morse}, we have a well-defined topology on $\Bm{} X_p$: the closed subsets $B \subset \Bm{} X_p$ are those satisfying the condition $$[x_n \in B, \forall n > 0 \text{ and } x_n \to x] \Rightarrow x \in B.$$

We show this topology is equivalent to a system of neighborhoods at a point in $\Bm{} X_p$.

\begin{lem} \label{lem:nbhds} Let $X$ be a proper geodesic space and $p \in X$. Let $\alpha \colon [0, \infty) \rightarrow X$ be an $N$-Morse geodesic ray with $\alpha(0)=p$ and for each positive integer $n$ let $V_n( \alpha)$ be the set of geodesics rays $\gamma$ such that $\gamma(0)= p$ and $d(\alpha(t), \gamma(t))< \delta_N$ for all $t<n$. Then $\{V_n(\alpha) \mid n \in \N \}$ is a fundamental system of (not necessarily open)  neighborhoods of $[\alpha]$ in $\Bm{} X_p$.
\end{lem}

\begin{proof}Need to show:
\begin{enumerate}
\item Each $V_i(\alpha) \in \{V_n(\alpha)  \}$ contains $\alpha$
\item If $V_i(\alpha), V_j(\alpha) \in \{V_n(\alpha) \mid n \in \N \}$ then there exists $V_k(\alpha) \in \{V_n(\alpha) \}$ such that $V_k(\alpha) \subset V_i(\alpha) \cap V_j(\alpha)$.
\item For each $V_i(\alpha)  \in \{V_n(\alpha) \}$ there exists a $V_j(\alpha)  \in \{V_n(\alpha)\}$ such that for each $\gamma \in V_j(\alpha)$ there exists $V_k(\gamma)  \in \{V_n(\gamma)\}$ such that $V_k(\gamma) \subset V_i(\alpha)$.
\item The topology on $\Bm{} X_p$ induced by the sequential definition and the fundamental system of  neighborhoods coincide.
\end{enumerate}
By Corollary \ref{asymp morse eventually close}, these sets determine well defined sets in $\Bm{} X$. We satisfy the first condition by definition. The second condition follows by setting $k= \max\{i, j \}$.

For the third condition consider a neighborhood $V_i(\alpha)$. Let $j=k= i+12\delta_N$. Let $\gamma \in V_j(\alpha)$ and $\gamma' \in V_k(\gamma)$. We know that $d(\alpha(t), \gamma'(t)) < d(\alpha(t), \gamma(t)) + d(\gamma(t), \gamma'(t)) < 2\delta_N$ for all $t \in [0, j]$. By Corollary \ref{morse close same basepoint} we know that $d(\alpha(t), \gamma'(t)) < \delta_N$ for all $t \in [0, i+8\delta_N]$ thus $V_k(\gamma) \subset V_i(\alpha)$.

To see the fourth condition we follow  \cite{bridson:1999fj} III.H Lemma 3.6 with $k=\delta_N$.
\end{proof}

\begin{cor}\label{maps continuous} Let $N$ and $N'$ be Morse gauges such that $N(\lambda,\epsilon) \leq N'(\lambda,\epsilon)$ for all $\lambda,\epsilon \in \N$. Then the obvious inclusion $i \colon \Bm{} X_p \hookrightarrow \Bm{'} X_p$ is continuous. 
\end{cor}

\begin{proof} Let $V$ be a closed set in $\Bm{'}X_P$. We wish to show $i^{-1}(V)$ is closed. Let $\alpha_i \in  i^{-1}(V)$ be a sequence of geodesic rays converging to a ray $\alpha$. Since $i$ is an inclusion we can consider $\alpha_i$ as a sequence of rays in $V$. Since $V$ is closed, the $\alpha_i$ converge to some ray $\alpha$ in $V$. But by Lemma \ref{M-Morse uniformly converge to M-Morse}, $\alpha$ is $N$-Morse and thus $i^{-1}(\alpha)= \alpha$ and thus $i^{-1}(V)$ is closed.
\end{proof}

With Corollary \ref{maps continuous} in mind we can now define the Morse boundary, $\bm X_p$. 

\begin{defn} Let $\mathcal M$ be the set of all Morse gauges. We put a partial ordering on $\mathcal M$ so that  for two Morse gauges $N, N' \in \mathcal M$, we say $N \leq N'$ if and only if $N(\lambda,\epsilon) \leq N'(\lambda,\epsilon)$ for all $\lambda,\epsilon \in \N$. Thus we can define  \begin{equation*}
\bm X_p=\varinjlim_\mathcal{M} \Bm{}X_p \end{equation*} with the induced direct limit topology. 
\end{defn}

\begin{rem}[Continuous maps between direct limits] \label{rem:direct lim} Let $X, Y$ be two proper geodesic metric spaces and $\bm X_p$ and $\bm Y_q$ be their Morse boundaries. Let $i^X_{N,N'} \colon \Bm{} X_p \to \Bm{'} X_p$ be the continuous inclusions as described in Corollary \ref{maps continuous}. Suppose $g \colon \mathcal{M} \to \mathcal{M}$ is a ``direction preserving map," i.e., $N \leq N' \implies g(N) \leq g(N')$. If for each $N \in \mathcal{M}$ we have a continuous map $f_N \colon \Bm{}X_p \to \Bmg Y_q$ such that $f_{N'} \circ i^{X}_{N,N'} =i^Y_{g(N),g(N')} \circ f_N$ whenever $N \leq N'$, then by the universal property of direct limits the family $\{f_N\}$ induces a continuous map $f \colon \bm X_p \to \bm Y_q$.
\end{rem}

\begin{prop}[Independence of basepoint] \label{prop:basept indep} Let $X$ be a proper geodesic space. The direct limit topology on $\bm X_p=\varinjlim\Bm{}X_p$ is independent of basepoint $p$. 
\end{prop}

\begin{proof}
Using Lemma \ref{basepoint change}, we see that there exists a map  \begin{equation*} i \colon \Bm{}X_p \hookrightarrow \Bmg X_{p'} \end{equation*} where $g \colon \mathcal{M} \to \mathcal{M}$ is a direction preserving map and $i(\alpha)$ is asymptotic to $\alpha$. This extends to a map $I \colon \bm X_p \to \bm X_{p'}$. We can do the same procedure and get a map $J \colon \bm X_{p'} \to \bm X_{p}$. We note that $J \circ I=Id$ because $J \circ I$ takes a Morse ray based at $p$ to another Morse ray based at $p$ that is asymptotic to the original. $I \circ J$ is also the identity map by the same reasoning and therefore $I$ is a bijection.

By Remark \ref{rem:direct lim} it is enough to show that $i \colon \Bm{}X_p \hookrightarrow \Bmg X_{p'}$ is continuous for all $N \in \mathcal{M}$. Let $\alpha \colon [0, \infty) \rightarrow X$ be a geodesic ray such that $\alpha(0)=p$. Let $\alpha'$ be a geodesic ray such that $i(\alpha)=\alpha'$ and consider some neighborhood $V_n(\alpha')\subset \Bmg X_{p'}$. Let $\theta= \max\{4N(1,2d(p,p')) + 3d(p,p'), \delta_N \}$ We claim if $m= n+6\theta$ then $i(V_m(\alpha)) \subset V_n(\alpha')$.

We know by Lemma \ref{basepoint change} that  $d(\alpha(t), \alpha'(t))< \theta$ for all $t \in [0, \infty)$. Let $\gamma$ be a ray in $V_m(\alpha)$. Then by definition, $d(\gamma(t), \alpha(t))< \delta_N \leq \theta$ for all $t \in [0, m]$. Consider $I(\gamma) =\gamma'$. Again, using Lemma \ref{basepoint change} we know that  $d(\gamma(t), \gamma'(t))< \theta$ for all $t \in [0, \infty)$. Thus 
\begin{align*}
d(\alpha'(t), \gamma'(t)) &<  d(\alpha(t), \alpha'(t)) + d(\alpha(t), \gamma(t))+ d(\gamma(t), \gamma'(t)) <3\theta 
\end{align*} for all $t \in [0,m]$. So by Corollary \ref{morse close same basepoint}, $d(\alpha'(t), \gamma'(t))< \delta_{g(N)}$ for all $t \in [0, n]$ and thus $i(\gamma) \in V_n(\alpha')$, and we have our result.
\end{proof}

\begin{rem}
In light of Proposition \ref{prop:basept indep}, when convenient, we will assume the basepoint is fixed, suppress it from the notation and write $\bm X = \varinjlim\Bm{}X$.
\end{rem}

Let $f: X \to Y$ be a $(\lambda, \epsilon)$-quasi-isometry. Fix base points $p \in X$ and $f(p) \in Y$. By Lemma \ref{straighten quasi-geodesics}, $f$ induces a map \begin{equation*} \bm f \colon \bm X_p \to \bm Y_{f(p)} \end{equation*} which maps $\Bm{} X_p$ into $\Bmg Y_{f(p)}$ for some direction preserving map $g: \mathcal{M} \to \mathcal{M}$.
 
\begin{prop}[Quasi-isometry invariance] \label{quasi-isometry invariance} Let $f \colon X \to Y$ be a $(\lambda, \epsilon)$-quasi-isometry of proper geodesic spaces. Then $\bm f \colon \bm X \to \bm Y$ is a homeomorphism. \end{prop}

\begin{proof} The proof that $\bm f$ is bijective is the same as in Theorem 3.11 in \cite{charney:2015aa}. It remains to prove continuity.

By Remark \ref{rem:direct lim} we need only show  \begin{equation*} \bm f_N \colon \Bm{}X \hookrightarrow \Bmg Y \end{equation*} is continuous. Let $\gamma \in \Bm{} X$ and consider  $V_n(\bm f(\gamma)) \subset \Bmg Y$. We show that there exists an $m$ sufficiently large such that $\bm f (V_m(\gamma)) \subset V_n(\bm f(\gamma))$. 

Let $\beta \in V_m(\gamma)$. Then $f \circ \beta$ and $f \circ \gamma$ are $N''$-Morse $(\lambda, \epsilon)$-quasi-geodesics. By definition of $V_m(\gamma)$ we know that \begin{equation*} d(f \circ \beta(t), f \circ \gamma(t))< \lambda \delta_N + \epsilon \end{equation*} for all $t \in [0,m]$. Moreover, by choosing $m$ sufficiently large, we may assume $(f \circ \beta)(m)$ and $(f \circ \gamma)(m)$ are arbitrarily far from the basepoint $f(p)$, say a distance $m' \gg n$. As in the proof of Lemma \ref{straighten quasi-geodesics}, we straighten $f \circ \beta$, $f \circ \gamma$ to $N'$-Morse geodesic rays $\beta' := \bm f(\beta)$, $\gamma' := \bm f(\gamma)$ which are Hausdorff distance $C$ from $f\circ \beta$ and $f \circ \gamma$ respectively (where $C$ depends only on $N, \lambda, \epsilon$). We note that $\beta'(s)$ lies in the $(\lambda \delta_{N} + \epsilon + 2C)$ neighborhood of the image of $\gamma'$ for some $s > m'-C$. Choosing $m' > n +(\lambda \delta_{N} + \epsilon + 2C) +4N'(3,0)+C$ we have by Lemma \ref{lem:twodelta} that $d(\beta'(t), \gamma'(t))< \delta_{g(N)}$ for all $t \in [0,n]$, and we have our result.
\end{proof}

Next we show that the Morse boundary coincides with the contracting boundary and the Gromov boundary. 

We begin with a description of the contracting boundary of a $\mathrm{CAT}(0)$ space. For more details see \cite{charney:2015aa}.

Let $X$ be a $\mathrm{CAT}(0)$ space. We define the set $\partial X$ to be the set of equivalence classes of geodesic rays up to asymptotic equivalence and denote the equivalence class of a ray by $[\alpha]$. One natural topology on $\partial X$ is the visual topology. We define the topology of the boundary with a system of neighborhood bases. A neighborhood basis for $[\alpha]$ is given by open sets of the form:
 $$U(\alpha, r, \epsilon) = \{[\beta] \in \partial X \mid \beta \text{ is a geodesic ray at } p \text{ and } \forall t<r, d(\beta(t), \alpha(t))< \epsilon \}.$$

\begin{defn}[contracting geodesics] Given a fixed constant $D$, a geodesic $\gamma$ is said to be \emph{$D$-contracting} if for all $x, y \in X$, $$d_X(x,y) < d_X(x, \pi_\gamma(x)) \implies d_X(\pi_\gamma(x), \pi_\gamma(y))<D.$$ We say that $\gamma$ \emph{contracting} if it is $D$-contracting for some $D$. An equivalent definition is that any metric ball $B$ not intersecting $\gamma$ projects to a segment of length $<2D$ on $\gamma$.
\end{defn}

Let $X$ be a complete $\mathrm{CAT}(0)$ space with basepoint $p \in X$. We define the \emph{contracting boundary} of a $\mathrm{CAT}(0)$ space $X$ to be the subset of the visual boundary consisting of $$\partial_c X_p = \{ [\alpha] \in \partial X \mid \alpha \text{ is contracting with basepoint } p \}.$$ 

In order to topologize the contracting boundary we consider a collection of increasing subsets of the boundary, $$\partial_c^n X_p = \{ [\gamma] \in \partial X \mid \gamma(0)=p, \gamma \text{ is an $n$-contracting ray}\},$$ one for each $n \in \N$. We topologize each $\partial_c^n X_p$ with the subspace topology from the visual boundary of $X$ and topologize the whole boundary by taking the direct limit over these subsets. Thus $\partial_c X_p=\varinjlim \partial_c^n X_p$ with the direct limit topology.

\begin{thm} If $X$ is a proper $\mathrm{CAT}(0)$ space then $\bm X$ and $\partial_c X$ are homeomorphic. \end{thm}

\begin{proof} 

Choose $p \in X$. By Theorem 2.9 in \cite{charney:2015aa} we know in a $\mathrm{CAT}(0)$ space that a geodesic ray is $D$-contracting if and only if it is $N$-Morse where $N$ depends only on $D$ and vice-versa. Thus it suffices to show that the topology on $\Bm{}X$ coincides with the subspace topology on $\partial_C^D X$ for the corresponding contracting constant $D$. It is clear that the topology on $\partial_C^D X$ contains the topology $\Bm{}X$. The reverse inclusion follows from the CAT(0) triangle condition. (Consider $U(\gamma, r, \epsilon)$. By the CAT(0) triangle condition we can choose an $n>>r$ large enough so that if $\beta \in V_n(\gamma)$ then $d(\gamma(t), \beta(t))< \epsilon$ for all $t<r$.)
 \end{proof}

\begin{thm} If $X$ is a proper geodesic $\delta$-hyperbolic space, then $\partial X = \bm X$. \end{thm}

\begin{proof} By \cite{bridson:1999fj} we know there exists a constant $R(K, L, \delta)$ such that if $\alpha$ is a $(K,L)$-quasi-geodesic  with endpoints on any geodesic $\gamma$, then $\alpha \in \mathcal{N}_R(\gamma)$. Setting $N(K,L)= R(K,L,\delta)$ we get $\bm X = \Bm{} X = \partial X$. The topology on $\partial X$, as defined in \cite{bridson:1999fj}, is identical to the topology on $\Bm{} X$. Thus we have a homeomorphism.
 \end{proof}
 
 At the other extreme, there are proper geodesic spaces where the Morse boundary is empty. Examples include products and groups with laws \cite{drutu:2005aa}.
 
Finally, we show that the Morse boundary is a visibility space. 
 
 \begin{prop}[Visibility] If $X$ is a proper geodesic metric space, then for each pair of distinct points $\xi_1, \xi_2 \in \bm X$ there exists a geodesic $\beta \colon \mathbb{R} \to X$ with $\beta([0,\infty))$ is asymptotic to $\xi_1$ and $\beta([0,-\infty))$ is asymptotic to $\xi_2$. Furthermore, $\beta$ is Morse where its Morse gauge depends on the Morse gauges of chosen representatives of $\xi_1$ and $\xi_2$. \end{prop}

\begin{proof} Choose a basepoint $p \in X$ and choose $N$-Morse geodesic rays $\alpha_1, \alpha_2 \colon [0,\infty) \to X$ with $\alpha_1(0)=\alpha_2(0)=p$ and $\alpha_1(\infty)= \xi_1$ and $\alpha_2(\infty) = \xi_2$.

Let $D \in \N$ be such that the distance from $\alpha_1(D)$ to the image of $\alpha_2$ is greater than $4N(3,0)$. For each $n > D$ consider a geodesic segment $\beta_n \colon [0, a] \to X$ with $\beta_n(0)=\alpha_1(n)$ and $\beta_n(a)=\alpha_2(n)$. 

By Lemma \ref{lem:triangle slim} the geodesic triangle $\alpha_1([0,n]) \cup \beta_n \cup \alpha_2([0,n])$ is $4N(3,0)$ slim. Thus $\beta_n$ must intersect a compact ball of radius $4N(3,0)$ at $\alpha_1(D)$ at a point $b_n \in \beta_n$.  By Corollary \ref{cor:AA2} there is a subsequence of $\{\beta_n\}$ which converges to a bi-infinite geodesic $\beta$. Since each $\beta_n$ is in the $4N(3,0)$ neighborhood of the images of $\alpha_1, \alpha_2$, then so is $\beta$. By a standard argument, the endpoints of $\beta$ are $\xi_1$ and $\xi_2$. 

To get the Morse conclusion we note by Lemma \ref{lem:third edge N'} that each $\beta_n$ is $N'$-Morse where $N'$ depends on $N$. Therefore by Lemma \ref{M-Morse uniformly converge to M-Morse} $\beta$ is $N'$-Morse.

\end{proof}

\begin{prop}  If $X$ is a proper geodesic space, then for any Morse gauge $N$, the $N$-Morse boundary, $\Bm{} X$, is compact.
\end{prop}

\begin{proof} For any Morse gauge $N$, $\Bm{} X$ is closed by Lemma \ref{M-Morse uniformly converge to M-Morse}. We note $\Bm{} X$ is first countable by the definition of the topology by countable neighborhood bases. It suffices to show that $\Bm{} X$ is sequentially compact, but this follows from Arzel\`a-Ascoli.
\end{proof}

\begin{rem} By the preceding proposition we know that $\Bm{} X$ is compact for all Morse functions $N$, but unlike in the case of the contracting boundary it does not follow that $\bm X$ is $\sigma$-compact. It is unknown whether you can define the same boundary as a direct limit over a countable subset of $\mathcal{M}$.
\end{rem}

\section{Morse preserving maps and application to Teichm\"uller space and the mapping class group}

\subsection{Morse preserving maps} \label{subsec:mpmaps}

\begin{defn} Let $X$ and $Y$ be proper geodesic metric spaces and $p \in X, p' \in Y$. We say that $\Omega \colon \bm X_p \to \bm Y_{p'}$ is \emph{Morse preserving} if given $N \in \mathcal{M}$ there exists an $N'\in \mathcal{M}$ such that $\Omega$ injectively maps $\Bm{}X_p \to \partial_M^{N'} Y_{p'}$.
\end{defn}

Proposition \ref{quasi-isometry invariance} shows that all quasi-isometries induce Morse preserving maps. Quasi-isometric embeddings, on the other hand, will not always induce Morse preserving maps. Consider the space $X$ formed by gluing a Euclidean half-plane to a bi-infinite geodesic $\gamma$ in the hyperbolic plane, $\Hyp$. The obvious embedding $\iota \colon \Hyp \hookrightarrow X$ is an isometric embedding, but $\iota(\gamma)$ is not Morse. 

Sometimes a quasi-isometric embedding $f\colon X \to Y$ does induce a Morse preserving map, i.e., given $N \in \mathcal{M}$ there exists an $N'\in \mathcal{M}$ such that for every $N$-Morse geodesic ray $\gamma\colon [0, \infty) \to X$ there exists an $N'$-Morse ray with basepoint $f(\gamma(0))$ bounded Hausdorff distance from $f(\gamma)$. (Equivalently, given $N \in \mathcal{M}$ there exists an $N'\in \mathcal{M}$ such that for every $N$-Morse geodesic ray $\gamma$, $f(\gamma)$ is an $N'$-Morse quasi-geodesic.) For example, if $X\subset Y$ is stable as defined by Durham--Taylor in \cite{durham:2015aa}, then the associated quasi-isometric embedding $X \to Y$ is always Morse preserving.


If $f \colon X \to Y$ is a quasi-isometric embedding that induces a Morse preserving map $\bm f$, we see that $\bm f$ mirrors the map used in Proposition \ref{quasi-isometry invariance} showing the quasi-isometry invariance of the Morse boundary. This suggests that $\bm X$ might be topologically embedded in $\bm Y$. The fact that $f$  \emph{uniformly} maps $N$-Morse rays close to $N'$-Morse rays is enough for $\bm f$ to be an injective continuous map. We wish to have a topological embedding, so we need to show that this map is open. As $f$ is only a quasi-isometric embedding, $f$ has no quasi-inverse, so we cannot exactly follow the proof of Proposition \ref{quasi-isometry invariance}. Nevertheless, in Proposition \ref{morse-preserving means embedding} we construct an inverse to $\bm f$ show this map is continuous.

\begin{prop} \label{morse-preserving means embedding} If $f\colon X \to Y$ is a quasi-isometric embedding that induces a Morse preserving map, then $$\bm f\colon \bm X \to \bm f(\bm X)$$ is a homeomorphism, i.e., $\bm f$ is a topological embedding.
\end{prop}

\begin{proof} 
By definition of Morse preserving, we know that for every $N \in \mathcal{M}$ there exists an $N'\in \mathcal{M}$ such that for every $N$-Morse geodesic ray $\gamma\colon [0, \infty) \to X$ there exists a $N'$-Morse ray with basepoint $f(\gamma(0))$ bounded Hausdorff distance from $f(\gamma)$. We define the map $$\bm f \colon \bm X \to \bm Y$$ as follows: if $\gamma$ is an $N$-Morse geodesic ray, then $\bm f(\gamma)$ is the $N'$-Morse geodesic ray bounded Hausdorff distance from $f(\gamma)$. Thus,  $f$ induces an injective map $\bm f_N \colon  \Bm{} X \to \Bmg{} Y$ where $g$ is direction preserving. We give $\bm f(\bm X)$ the subspace topology inherited from the topology on $\bm Y$.  To show that $\bm f$ is continuous, by Remark \ref{rem:direct lim} we need only show $\bm f_N \colon  \Bm{} X \to \Bmg{} Y$ is continuous. This follows with slight modification of the proof of Proposition \ref{quasi-isometry invariance} by intersecting the appropriate basis neighborhoods with $\bm f(X)$. 

Since $\bm f$ is injective it has a (set-theoretic) inverse $h \colon \bm f (\bm X) \to \bm X$. So, to show the result, we need only show that $h \colon \bm f(\bm X) \cap \Bm{}Y \to \Bmprime X$ is continuous. We know that for any $\eta \geq 0$ we have a quasi-isometry $$f_{\eta} \colon X \to \mathcal{N}_\eta(f(X)).$$ Let $h_\eta \colon \mathcal{N}_\eta(f(X)) \to X$ be a quasi-inverse. As in the proof of Lemma \ref{straighten quasi-geodesics}, we know that given a Morse function $N$, there exists $\eta$ which depends on $\lambda, \epsilon$ and an $N'$ such that $\bm f(\bm X) \cap \Bm{} Y \subset \Bm{'} \mathcal{N}_\eta(f(X))$. Thus, we get a map  $\bm h_\eta \colon \bm f(\bm X) \cap \Bm{}Y \to \Bmprime X$. Again, with slight modification of the proof of Proposition \ref{quasi-isometry invariance} we note that $\bm h_\eta$ is continuous. We notice that for any $\alpha \in \bm X$, $h_\eta \circ f(\alpha)$ lies bounded distance from $\alpha$, so $\bm h_\eta(\bm f (\alpha)) = \alpha$. That is, $\bm h_\eta=h$ on $\bm f (\bm X) \cap \Bm{}Y$. We conclude that $\bm f$ is a topological embedding.
\end{proof}

\subsection{Application to Teichm\"uller space}

We now describe the Morse boundary of Teichm\"uller space with the Teichm\"uller metric. 

Let $S$ be a surface of finite type. We know from Theorem 4.2 in  \cite{minsky:1996aa} that if $\gamma$ is a geodesic contained in the $\epsilon$-thick part of Teichm\"uller space, $\thick$, that any $(K,L)$-quasi-geodesic with endpoints on $\gamma$ remains in an $N(K, L, \epsilon)$-neighborhood of $\gamma$.  By Theorem 5.2 of \cite{minsky:1996aa} we know that if a geodesic is not contained in $\thick$ for any $\epsilon>0$ then it does not have the Morse property.  It follows that given a Morse function $N$, there exists $\epsilon$ and a Morse function $N'$ such that \begin{equation}\label{bmteich} \Bm{} \teich \subset \partial \thick \cap \bm \teich \subset \Bmprime \teich \end{equation} where $\partial \thick$ is the set of geodesic rays in $\teich$ (up to asymptotic equivalence) whose image lies in $\thick$. By abuse of notation set $\bm \thick \colon= \partial \thick \cap \bm \teich$. Then (\ref{bmteich}) implies that $$\bm \teich= \varinjlim_\epsilon \bm \thick.$$

The main theorem of Leininger and Schleimer in \cite{leininger:2014aa} states that for all $n \in \N$, there exists a surface $S$ of finite type and a quasi-isometric embedding $$\Omega \colon \Hyp^n \to \teich.$$ Moreover, the image is quasi-convex and lies in $\thick$ for some $\epsilon >0$. We show that this map is Morse preserving.

\begin{prop} \label{teich Morse-preserving} There exists an $N \in \mathcal{M}$ such that if $\alpha \colon [0, \infty) \to \Hyp^n$ is any geodesic ray with basepoint $p \in \Hyp^n$, then the quasi-geodesic $\Omega(\alpha) \colon [0, \infty) \to \teich$ is bounded Hausdorff distance from an $N$-Morse geodesic ray $\beta$ where $\beta(0)=\Omega(p)$, i.e., $\Omega$ is Morse preserving.
\end{prop}

\begin{proof} 
Say $\Omega \colon \Hyp^n \to \teich$ is a $(K,L)$-quasi-isometric embedding and $\Omega(\Hyp^n)$ is $C$-quasi-convex and lies in $\thick$. Then, as in the proof of Corollary 2.7 in \cite{Hamenstadt:2010aa}  $\mathcal{N}_C(\Omega(\Hyp^n)) \subset \thickprime$ for some $\epsilon>\epsilon'$. 

Let $\alpha$ be a geodesic in $\Hyp^n$, and $\{\beta_n\}_{n \in \N}$ be a sequence of geodesics joining $\Omega(p)$ and $\Omega(\alpha(n))$. Then by quasi-convexity of $\Omega(\Hyp^n)$, each $\beta_n \subset \thickprime$ and thus by Theorem 4.2 in  \cite{minsky:1996aa} we know that all the $\beta_n$ are $N$-Morse where the Morse function $N$ depends on $\epsilon'$. In particular, by Lemma \ref{near morse means hausdorff close} we know $\beta_n$ and $\Omega(\alpha)|_{[0,n]}$ have bounded Hausdorff distance $D=2N(K,2(K+L))+(K+L)$. Thus by Arzel\`a-Ascoli there exists a subsequence $\beta_{n(i)}$ that converges to an $N$-Morse geodesic ray $\beta$ that is at most Hausdorff distance $D$ from $\Omega(\alpha)$.
\end{proof}

\begin{cor} For any $n \geq 2$, there exists a surface of finite type $S$ such that $\bm \teich$ contains a topologically embedded $S^{n-1}$.
\end{cor}

\begin{proof} This follows directly from Propositions \ref{morse-preserving means embedding}  and \ref{teich Morse-preserving} and the fact that $\bm \Hyp^n$ is homeomorphic to $S^{n-1}$.
\end{proof}

\subsection{Morse boundary of the mapping class group}
We now show that the Morse boundary of the mapping class group of a surface $S$ of finite type, $\Mod$, and the Teichm\"uller space of that surface, $\teich$, are homeomorphic.

We assume the reader is familiar with the Masur-Minsky machinery developed in \cite{Masur:2000aa} and the coarse geometry of Teichm\"uller space. We will quickly review the basic notions. The \emph{curve graph}, $\mathcal{C}(S)$, is a locally infinite simplicial graph whose vertices are isotopy classes of essential simple closed curves on $S$, and we join two isotopy classes of curves if there exist representatives of each class which are disjoint. 

A complete clean \emph{marking} on $S$ is a set $\mu=\{(\alpha_1, \beta_1), \ldots, (\alpha_m, \beta_m)\}$ where $\{\alpha_1, \ldots, \alpha_m\}$ is a pants decomposition of $S$, and each $\beta_i$ is disjoint from $\alpha_j$ for $i \neq j$ and intersects $\alpha_i$ once (twice) if the surface filled by $\alpha_i$ and $\beta_i$ is a once-punctured torus (four-times-punctured sphere). We call the $\alpha$ curves the \emph{base} of $\mu$ and for every $i$, $\beta_i$ is called the \emph{transverse curve to $\alpha_i$ in $\mu$}.  The marking graph, $\MC$, is a graph whose vertices are (complete clean) markings and two markings $\mu_1, \mu_2 \in \mathrm{V}(\MC)$ are joined by an edge if they differ by an elementary switch. 

Given a nonannular subsurface $Y \subset S$ and a curve $\alpha \in \mathcal{C}(S)$ we define the \emph{subsurface projection of $\alpha$ to $Y$} to be the subset $\pi_Y(\alpha) \subset \mathcal{C}(Y)$ defined by taking the arcs of the intersection of $\alpha$ with $Y$ and performing surgery on the arcs to obtain closed curves in $Y$. (See \cite{Masur:2000aa} for a more precise definition and the definition in the annular case.) When dealing with markings, one only projects the base of the marking. We define the subsurface projection distance of two marking $\mu_1, \mu_2 \in \MC$ as $d_Y(\mu_1, \mu_2) = \mathrm{diam}_{\mathcal{C}(Y)}(\pi_Y(\mu_1) \cup \pi_Y(\mu_2)).$ We note that if $Y$ is the whole surface $S$, then $d_S$ is the usual distance in $\mathcal{C}(S)$.

One useful property of the marking graph is that it is quasi-isometric to $\Mod$. The quasi-isometry between the two spaces is defined by choosing a marking $\mu \in \mathrm{V}(\MC)$ and considering the mapping class group orbit. Thus, keeping in mind Theorem \ref{quasi-isometry invariance}, for the remainder of the section we will use $\MC$ in place of $\Mod$. 

For each $x \in \teich$ there is a \emph{short marking}, which is constructed inductively by picking the shortest curves for the base and repeating for the transverse curves.

It is a well known fact that there is a coarsely well-defined map $\Upsilon \colon \MC \to \teich$ that is coarsely Lipschitz. The map is defined by taking a marking to the region in the thick part of $\teich$ where the marking is a short marking for the points in that region.

Given a surface $Y$ with compact boundary for which the interior of $Y$ is a surface of genus $g$ with $p$ punctures we define the complexity of $Y$, $\xi(Y)=3g+p$.

We now list a collection of theorems that we will use to construct the homeomorphism between $\bm \teich$ and $\bm \Mod$. Theorem \ref{Rafi} and Theorems \ref{DT} show that geodesics with uniformly bounded subsurface projections are Morse and vice versa. Theorem \ref{comb-dist} (combinatorially) quantifies the relationship between subsurface projections and distances in $\MC$ and $\teich$.

\begin{thm}\label{Rafi} For every $E>0$ there exists an $\epsilon>0$ such that if $\gamma \colon [a,b] \to \teich$ is a Teichm\"uller geodesic and $\mu_a, \mu_b$ are short markings of $\gamma(a)$ and $\gamma(b)$ with $d_Y(\mu_a, \mu_b)<E$ for every proper domain $Y \subset S$ with $\xi(Y) \neq 3$, then $\gamma$ is $\epsilon$-cobounded.

Conversely for every $\epsilon >0$ there exists an $E>0$ such that if $\gamma \colon [a,b] \to \teich$ is an $\epsilon$-cobounded Teichm\"uller geodesic, then for short markings  $\mu_a, \mu_b$ of $\gamma(a)$ and $\gamma(b)$ $d_Y(\mu_a, \mu_b)\leq E$ for every proper domain $Y \subset S$ with $\xi(Y)\neq 3$. 
\end{thm}

\begin{proof} Implicit in the proof of Theorem 1.5 in \cite{Rafi:2005aa}\end{proof}

\begin{thm} \label{DT} \label{heir-paths} Let $\gamma \colon [a,b] \to \MC$ be an $N$-Morse geodesic then there exists an $E>0$ depending on $N$ such that for any two markings $\mu_1, \mu_2$ on the image of $\gamma$, $d_Y(\mu_1, \mu_2)<E$ for every proper domain $Y \subset S$ with $\xi(Y)\neq 3$.

Conversely if $\gamma \colon [a,b] \to \MC$ is a geodesic with $\gamma(a)= \mu_1$ and $\gamma(b)= \mu_2$ such that there exists an $E >0$ so that $d_Y(\mu_1, \mu_2)<E$ for every proper domain $Y \subset S$ with $\xi(Y)\neq 3$ then $\gamma$ is an $N$-Morse where $N$ depends on $E$.
\end{thm}

\begin{proof} The first statement is implicit in the proof of Theorem 6.3 in \cite{durham:2015aa}. 

To prove the second statement assume $\gamma \colon [a,b] \to \MC$ is a geodesic such that the endpoints have uniformly bounded proper subsurface projections. We know by Corollary 8.3.4 of \cite{Durham:2013aa} that there is a hierarchy path $H \colon [0,D] \to \MC$ with $H(0)= \mu_1$ and $H(D)=\mu_2$ which is $(K,L)$-quasi-geodesic. We know following the logic in Theorem 5.6 in  \cite{durham:2015aa} that $H$ is $N'$-Morse where $N'$ depends on $E$. (This also follows from Theorem 4.4 and the remark after Theorem 4.3 in \cite{Brock:2011aa}.) Thus $\gamma$ is a $(1,0)$-quasi-geodesic with endpoints on $H$ and thus by Lemma \ref{near morse means hausdorff close} and by Lemma 2.5 (1) of \cite{charney:2015aa} $\gamma$ is $N$-Morse where $N$ depends on $N'$.
\end{proof}

Given two functions $f, g$ the notation $f(x) \asymp g(x)$ means there exists constants $K \geq 1, L \geq 0$ such that $\frac{1}{K} g(x) -L \leq f(x) \leq Kg(x)+L$.

\begin{thm}[Combinatorial distance formulae] \label{comb-dist} There is a constant $A_0 > 0$ depending only on $S$, so that for any $A \geq A_0$
\begin{enumerate}
\item \label{mark-comb-dist}  and for any pair of markings $\mu_1, \mu_2 \in \MC$, we have $$d_{\MC}(\mu_1, \mu_2) \asymp \sum_{Y \subset S} [d_Y(\mu_1, \mu_2)]_A$$
where the sum is over all subsurfaces $Y\subset S$ and $[X]_A=X$ if $X \geq A$ and 0 otherwise

\item \label{teich-comb-dist} and for any pair $\sigma_1, \sigma_2 \in \thick$ with $\mu_1$ and $\mu_2$ short markings on $\sigma_1$ and $\sigma_2$, respectively, we have $$d_{\teich}(\sigma_1, \sigma_2) \asymp \sum_{Y} [d_Y(\mu_1, \mu_2)]_A$$ 
\end{enumerate}
where the sums are over all subsurfaces $Y\subset S$ (but for annular surfaces $B\subset S$, the distance $d_B$ is measured in $\mathbb{H}^2$) and $[X]_A=X$ if $X \geq A$ and 0 otherwise.\end{thm}

\begin{proof} (\ref{mark-comb-dist}) is Theorem 6.12 in \cite{Masur:2000aa} and (\ref{teich-comb-dist}) 
is Proposition A.1 in \cite{dowdall:2014aa}, a refinement of Rafi's distance formula \cite{Rafi:2007aa},   (see also \cite{Durham:2013aa}).
\end{proof}

\begin{lem} \label{lem-teich-delta-slim} Let $\alpha_1$ and $\alpha_2$ be $N$-Morse geodesic rays in Teichm\"uller space with $\alpha_1(0)=\alpha_2(0)$ and let $x \in \alpha$ and $y \in \beta$. Then any geodesic $\gamma\colon [a,b]\to \teich$ joining $x$ and $y$ is $N'$-Morse where $N'$ depends on $N$.
\end{lem}

\begin{proof} This follows from Lemma \ref{lem:third edge N'}.  
\end{proof}

Together the next two lemmas show that the distance between two points on a Morse geodesic in either $\MC$ or $\teich$ is coarsely the distance between the associated markings in the curve graph.

\begin{lem} \label{orbit-map-is-quasi-geo}  
If $\alpha \colon \N \to \MC$ is an $N$-Morse geodesic then $\Upsilon(\alpha) \colon \N \to \teich$ is an $N'$-Morse $(A,B)$-quasi-geodesic where $N',A,B$ depends on $N$.
\end{lem}

\begin{proof} 
By Theorem \ref{DT} since $\alpha$ is $N$-Morse  there exists an $E> 0$ such that for all $i, j \in \N$ where $i \neq j$, $d_Y(\alpha(i), \alpha(j))<E$ for every proper domain $Y \subset S$. Let $\Upsilon(\alpha(i)) = \sigma_i$ and let $\alpha(i)=\mu_i$. We consider equations (\ref{mark-comb-dist}) and (\ref{teich-comb-dist}) in Theorem \ref{comb-dist} and choose $A=\max\{A_0,E\}$. Since we have only uniformly bounded \emph{proper} subsurfaces we see that $$d_{\MC}(\mu_i, \mu_j) \asymp d_S(\mu_i, \mu_j) \asymp d_{\teich}(\sigma_i, \sigma_j),$$ i.e., coarsely the distance in the curve graph. We note that the constants in both $\asymp$ only depend on $N$. 

For any two points  $\Upsilon(\mu_i),\Upsilon(\mu_j)$ the geodesic $[\Upsilon(\mu_i),\Upsilon(\mu_j)]$ is contained in $\thick$ for some $\epsilon$ by Theorem \ref{Rafi} and thus is $N''$-Morse where $N''$ depends on $\epsilon$ by \cite[Theorem 4.2]{minsky:1996aa}. Lemma \ref{near morse means hausdorff close} tells us that the Hausdorff distance between $\Upsilon([\mu_i, \mu_j])$ and $[\Upsilon(\mu_i),\Upsilon(\mu_j)]$ is bounded by a constant $C$ where $C$ depends only on $N''$. It follows that $\Upsilon(\alpha)$ is $N'$-Morse where $N'$ depends on $N$.

\end{proof}

\begin{lem} \label{teich-shadow-is-quasi}
Let $\alpha  \colon [0, \infty) \to \teich$ be an $N$-Morse geodesic.  For each $i \in \N$ let $\mu_i$ be a short marking of $\alpha(i)$. Then the map $\gamma \colon \N \to \MC$ defined by $\gamma(i)=\mu_i$ is an $N'$-Morse $(A,B)$-quasi-geodesic where $N', A,B$ depend on $N$.
\end{lem}

\begin{proof} 
Since $\alpha$ is $N$-Morse, its image lies in $\thick$ for some $\epsilon>0$ \cite{minsky:1996aa}. By Theorem \ref{Rafi} we know for any $i \neq j$, $d_Y(\mu_i, \mu_j)<E$ for every proper domain $Y \subset S$. As in Lemma \ref{orbit-map-is-quasi-geo} choose $A=\max\{A_0,E\}$. Since we have only uniformly bounded \emph{proper} subsurfaces we see that $$d_{\MC}(\mu_i, \mu_j) \asymp d_S(\mu_i, \mu_j) \asymp d_{\teich}(\alpha(i), \alpha(j)).$$ We note that the constants in both $\asymp$ only depend on $N$. 

For any two points $\gamma(i), \gamma(j)$ by Theorem \ref{DT} we know that any geodesic $[\gamma(i), \gamma(j)]$ is $N''$-Morse where $N''$ depends on $E$. Lemma \ref{near morse means hausdorff close} tells us that the Hausdorff distance between $\gamma([i,j])$ and $[\gamma(i), \gamma(j)]$ is bounded by a constant $C$ where $C$ depends only on $N''$. It follows that $\gamma(\alpha)$ is $N'$-Morse where $N'$ depends on $N$.
\end{proof}

\begin{prop} \label{teich-pmf-bijection} There is a natural bijection between $\bm \MC$ and $\bm \teich$.
\end{prop}

\begin{proof} 
Fix a basepoint $p \in \MC$.

We first define a map $f \colon \bm \MC \to \bm \teich$. Let $\alpha\colon \N \to \MC$ be an $N$-Morse geodesic with $\alpha(0)=p$. By Lemma \ref{orbit-map-is-quasi-geo} we know that  $\Upsilon(\alpha) \colon \N \to \teich$ is an $N'$-Morse quasi-geodesic whose quasi-constants and $N'$ depend on $N$. Let $\beta_n$ be a geodesic joining $\Upsilon(\alpha(0))$ and $\Upsilon(\alpha(n))$. By Theorem \ref{Rafi} we know that there exists $\epsilon>0$ such that $\beta_n$ is $\epsilon$-thick for every $n \in \N$ and thus by \cite{minsky:1996aa} the $\beta_n$ are all $N''$-Morse where $N''$ depends on $N$. By Arzel\`a-Ascoli there is a subsequence $\beta_{n(i)}$ that converges to a geodesic ray $\beta$ which is $N''$-Morse by Lemma \ref{M-Morse uniformly converge to M-Morse}. We define $f(\alpha)=\beta$. We note that by Lemma \ref{orbit-map-is-quasi-geo} we know that $\beta$ is bounded Hausdorff distance from $\Upsilon(\alpha)$ where the bound depends only on $N$. 

We now define a map $g \colon \bm \teich \to \bm \MC$. Let $\beta \colon [0, \infty) \to \teich$ be an $N$-Morse geodesic ray with basepoint $\Upsilon(p)$. Following the shortest marking map defined in Lemma \ref{teich-shadow-is-quasi} we get an $N'$-Morse quasi-geodesic $\gamma$ whose quasi-constants and $N'$ depend on $N$. Let $\alpha_n$ be a geodesic joining $\gamma(0)$ and $\gamma(n)$. By Lemma 2.5 (1), (2) of \cite{charney:2015aa} we know that the geodesic, $\alpha_n$ is $N''$-Morse where $N''$ depends on $N$. As above, we use Arzel\`a-Ascoli to extract a subsequence ${\alpha}_{n(i)}$ which converges to an $N''$-Morse geodesic ray $\alpha$. We define $g(\beta)=\alpha$. We can conclude that $\alpha$ is bounded Hausdorff distance from $\gamma$ where the bound depends only on $N$. 

To show that $f$ is a bijection we will show that $f \circ g$ is the identity and $f$ is injective.

We first show $f \circ g$ is the identity. Let $\beta \colon [0, \infty) \to \teich$ be an $N$-Morse geodesic ray with basepoint $\Upsilon(\alpha(0))$. As above, we know that $g(\beta)$ is a geodesic ray $\alpha$ that is bounded Hausdorff distance, from $\gamma$, the image of the shortest marking map. We wish to show that $f(g(\beta))=f(\alpha)$ and $\beta$ have bounded Hausdorff distance.  Since $\Upsilon$ is coarsely Lipschitz, we know that $\Upsilon(\alpha)$ and $\Upsilon(\gamma)$ have bounded Hausdorff distance. By construction $\Upsilon(\gamma)$ and $\beta$ have bounded Hausdorff distance. Also by construction, $f(\alpha)$ and $\Upsilon(\alpha)$ have bounded Hausdorff distance. Putting all these together we get that $f(\alpha)$ and $\beta$ have bounded Hausdorff distance, and we have our result.

We now show $f$ is injective. Let $\alpha_1, \alpha_2 \colon \N \to \MC$ be Morse geodesics such that $f(\alpha_1)=f(\alpha_2)$. Then for some $K>0$ we have $d(f(\alpha_1)(i), f(\alpha_2)(i))< K$ for all $i \in \N$. We know that the proper subsurface projections of $f(\alpha_1)$ and $f(\alpha_2)$ are uniformly bounded because they are both Morse. Thus, the inequality in the proof of Lemma \ref{orbit-map-is-quasi-geo} forces the $d(\alpha_1(i), \alpha_2(i))<KA+AB$ for all $i \in \N$ which means $\alpha_1=\alpha_2$. 
\end{proof}

\begin{thm} Let $S$ be a surface of finite type. Then the Morse boundary of $\teich$ is homeomorphic to the Morse boundary of $\Mod$.
\end{thm}

\begin{proof}

After Proposition \ref{teich-pmf-bijection}, what is left to show is continuity of $f$ and $g$.

The continuity of $f\colon \bm \MC \to \bm \teich$ follows as in the proof of Theorem \ref{quasi-isometry invariance} because the $\Upsilon$ is coarse Lipschitz and the fact that if $\alpha$ is an $N$-Morse ray in $\MC$ that $\Upsilon(\alpha)$ is bounded Hausdorff distance from $f(\alpha)$ and that bound only depends on on $N$. 

We now prove the continuity of $g \colon \bm \teich \to \MC$. Let $\beta_1, \beta_2 \colon [0, \infty) \to \teich$ be two $N$-Morse rays with $\beta_1(0)=\beta_2(0)=p$. By Lemma \ref{lem-teich-delta-slim} we know that for any $i,j \in \N$ the geodesic joining $\beta_1(i)$ and $\beta_2(j)$ is $O$-Morse where $O$ depends on $N$. Let $\mu_1(i)$ and $\mu_2(j)$ be short markings of $\beta_1(i)$ and $\beta_2(j)$  respectively. As in the proof of Lemma \ref{orbit-map-is-quasi-geo}  we see for any $i,j \in \N$ that $d_{\teich}(\beta_1(i), \beta_2(j)) \asymp  d_{\MC}(\mu_1(i), \mu_2(j)).$ In particular there exist $A,B$ depending on $N$ such that $$d_{\MC}(\mu_1(i), \mu_2(j)) \leq A\cdot d_{\teich}(\beta_1(i), \beta_2(j))+B.$$ The continuity of $g \colon \bm \teich \to \MC$ follows as in the proof of Theorem \ref{quasi-isometry invariance} from the inequality above.

Thus we have shown that the Morse boundary of $\teich$ is homeomorphic to the Morse boundary of $\Mod$.

\end{proof}

\subsection{A continuous injective map from $\bm \teich$ to $\pmf$} 
Let $\mathrm{MF}(S)$ be the set of measured foliations on $S$, and let $\pmf$  be the space of projective measured foliations on $S$. The Thurston compactification of Teichm\"uller space is  $\overline{\teich}=\teich \cup \pmf$. Recall that Teichm\"uller's theorem says that for any $x \neq y \in \teich$ there exists a unique pair $(\xi, \eta) \in \pmf \times \pmf$ such that $x$ and $y$ lie on the Teichm\"uller geodesic $\overleftrightarrow{(\xi, \eta)}$. (For more information see \cite{papadopoulos:2007aa}.) In general, for $x \in \teich$ and $\eta \in \pmf$ it is not know if the geodesic ray $\gamma=\overrightarrow{[x, \eta)}$ converges in Thurston's compactification $\overline{\teich}$. In the case when $\eta \in \pmf$ is uniquely ergodic, then a theorem of Masur says $\gamma$ converges to $\eta$ \cite{masur:1982ab}. Masur also proves that the ``endpoints" of Morse geodesic rays in $\teich$ are uniquely ergodic \cite{Masur:1982aa}. Furthermore he shows that when $\eta$ is uniquely ergodic, any two rays with direction $\eta$ in Thurston's compactification are asymptotic \cite{masur:1980aa}. Thus we have a map from $h_\infty \colon \bm \teich \to \pmf$. The next lemma says this map is injective and well-defined.

\begin{lem} Let $\overrightarrow{[x, \xi)}, \overrightarrow{[y, \eta)}$ be two Morse rays in $\teich$. The rays $\overrightarrow{[x, \xi)}, \overrightarrow{[y, \eta)}$ have finite Hausdorff distance in $\teich$ if and only if $\xi = \eta$.
\end{lem}

\begin{proof} The forward direction is Lemma 2.4 in \cite{farb:2002aa}. For the other direction we apply the result of Masur mentioned above \cite{masur:1980aa}. \end{proof}

\begin{prop} The injective map $h_\infty \colon  \bm \teich \hookrightarrow \pmf$ is continuous.
\end{prop}

\begin{proof} We show this by showing for each Morse gauge $N \in \mathcal{M}$ the map $$h_N \colon \Bm{}\teich \to \pmf$$ is a topological embedding. We follow closely the explication in Facts 3 and 4 in Farb and Mosher's description of the ``limit set" in  \cite{farb:2002aa}. Let $p \in \teich$.

We first show that for each Morse gauge $N$, $\Bm{} \teich_p$ is a closed subset of $\pmf$, and therefore compact. Choose a sequence $[\alpha_i] \in\Bm{} \teich_p$ so that $\lim [\alpha_i] =[\alpha_\infty]$ in $\pmf$. It is enough to show that $[\alpha_\infty] \in \Bm{} \teich_p$. Let $\overrightarrow{[p, \eta_i)}$ be $N$-Morse representatives of  $[\alpha_i]$. By Arzel\`a-Ascoli we can pass to a subsequence that converges to a ray $\lim \overrightarrow{[p, \eta_i)} = \overrightarrow{[p, \eta'_\infty)}$. Since this is a sequence of $N$-Morse rays, we know that $\overrightarrow{[p, \eta'_\infty)} \in \Bm{} \teich_p$. We look at the unit tangent bundle of $\teich$ at the point $p$ and it follows that $\lim \eta_i=\eta'_\infty$, and so we have $\eta_\infty = \eta'_\infty \in  \Bm{} \teich_p$.

We now show $h_N \colon \Bm{}\teich \to \pmf$ is a homeomorphism onto its image. Since both the domain and range are compact Hausdorff spaces, it suffices to prove continuity in one direction. We prove that $h_N^{-1}$ is continuous. This follows from the observation that for a convergent subsequence $\eta_i \to \eta$ in $h_N( \Bm{} \teich_p) \subset \pmf$, the sequence of rays $\overrightarrow{[p, \eta_i)}$ converges in the compact open topology to the ray $\overrightarrow{[p, \eta)}$.
\end{proof}

\appendix

\section{Corrigendum}
\subsection{Introduction}
In Lemma \ref{M-Morse uniformly converge to M-Morse} the first author claimed that if a sequence of $N$-Morse geodesic rays converged uniformly on compact sets to a geodesic ray, then that ray would be $N$-Morse. This does not follow and we provide a counterexample below (Example \ref{ex: counterexample}). This corrigendum will do three things: We first introduce the notion of refined Morse gauge and (quasi)-geodesics and use them to prove a corrected version of Lemma \ref{M-Morse uniformly converge to M-Morse}. We then reprove Corollary \ref{maps continuous} without using Lemma \ref{M-Morse uniformly converge to M-Morse} so that one can define the Morse boundary without reference to refined Morse gauges. Finally we show that for any proper geodesic metric space, the Morse boundary defined with Morse gauges and the refined Morse boundary, defined by refined Morse gauges only, are homeomorphic. A key step in this is Lemma \ref{lem:every gauge gives refined}, which associate to any Morse gauge $N$ a refined Morse gauge $\hat N$ such that all $N$-Morse geodesics are $\hat N$-Morse; essentially the lemma says that one can always replace Morse gauges with refined Morse gauges.

A major consequence of the incorrectness of Lemma 2.10 is that the $N$-Morse strata, $\partial_M^N X$, are not necessarily compact. Many papers which use the Morse boundary rely on Lemma 2.10 either directly or by using Theorem 3.14 of \cite{Cordes:2019ab}. Thankfully, by Theorem \ref{thm:homeo} anywhere where these results are called upon, unless relying specifically on Morse gauges that are not refined, one may simply substitute Morse gauges for refined Morse gauges and the same conclusion will follow. For instance, Theorem 3.14 of \cite{Cordes:2019ab} holds if one considers only refined Morse gauges.  

\subsection{Refined Morse boundary}
We now introduce the notion of refined Morse gauge and geodesic:

\begin{defn}
A refined Morse gauge is a Morse gauge $N \colon \mathbb{R}_{\geq 1} \times \mathbb{R}_{\geq 0} \to \mathbb{R}_{\geq 0}$ with the following additional properties:
\begin{enumerate}
    \item $N$ is non-decreasing
    \item $N$ is continuous in the second coordinate
\end{enumerate}
We denote the collection of refined Morse gauges $\widehat{\mathcal{M}}$.
\end{defn}

To correct the proof of Lemma 2.10, we do not actually need that a Morse gauge $N$ is non-decreasing, but it is generally useful. In particular, given a Morse gauge $N$ one has an associated constant $\delta_N$ (to be thought of as a hyperbolicity constant) which is useful in several ways. If $N\leq N'$ are non-decreasing Morse gauges, then $\delta_N\leq \delta_{N'}$, which does not hold in general otherwise (unless one redefines $\delta_N$ as proposed in Remark \ref{rem:deltan}). Furthermore, it is an intermediate step in the proof of Lemma \ref{lem:every gauge gives refined}. We now define a refined Morse quasi-geodesic.

\begin{defn}
Let $X$ be a metric space, $N$ a refined Morse gauge, and $I$ a closed interval of $\mathbb{R}$. The quasi-geodesic $\gamma \colon I \to X$ is a \emph{refined $N$-Morse geodesic} if for any $(\lambda, \epsilon)$-quasi-geodesic $\sigma$ with endpoints on $\gamma$, we have the image of $\sigma$ is contained in the closed $N(\lambda, \epsilon)$-neighborhood of $\gamma$.
\end{defn}

There are two differences between the definition of Morse geodesics and refined-Morse geodesics. The first is that we require the use of a \emph{refined} Morse gauge. The second is that we ask for the quasi-geodesic $\sigma$ to be in the \emph{closed} neighborhood of $\gamma$. 

We now present the counterexample to Lemma \ref{M-Morse uniformly converge to M-Morse}:
\begin{ex} \label{ex: counterexample}
Consider the space $X$ formed by taking the hyperbolic plane $\mathbb{H}^2$ and gluing in a line segment connecting two points $p,q \in \mathbb{H}^2$ of length $d(p,q)$. Thus in $X$ there are two geodesics between $p$ and $q$, one in $\mathbb{H}^2$ and the other along the line segment. Let $\alpha$ be a geodesic ray with basepoint $p$ passing through $q$ whose image is in $\mathbb{H}^2$ and let $\{\alpha_i\}$ be a collection of geodesic rays with image in $\mathbb{H}^2$ and basepoint $p$ converging to $\alpha$. Since $X$ is hyperbolic, every geodesic is uniformly $N$-Morse for some $N$. And because of the extra segment we glued on, $N(1,0)  \neq 0$. Define 
\[ N'(\lambda, \epsilon) = 
    \begin{cases} 
      0 & \lambda =1, \, \epsilon =0 \\
      N(\lambda, \epsilon) & \text{else} 
   \end{cases}
\]
Since $X$ is uniquely geodesic except for geodesics passing through $p$ and $q$, we know that the $\alpha_i$ are $N'$-Morse. Since $\alpha$ is not $N'$-Morse we have our counterexample. 

The Morse gauge $N'$ is not a refined Morse gauge because it is not continuous in the second coordinate. To see this, we note that $N'(1, \epsilon)=N(1,\epsilon)$ must be at least $\frac{1}{2}d(p,q)$ since a geodesic is a $(1, \epsilon)$-quasi-geodesic. So $N'$ is discontinuous at $0$. 
\end{ex}

The following lemma can be used in several contexts to harmlessly pass from Morse gauges to refined Morse gauges.

\begin{lem} \label{lem:every gauge gives refined}
%
For every Morse gauge $N$, there exists a refined Morse gauge $\widehat{N}$ with the property that every $N$-Morse geodesic is also an $\widehat{N}$-Morse geodesic.
\end{lem}

\begin{proof}
We will first construct a non-decreasing Morse gauge from $N$. This is done by setting 
\[
N'(\lambda, \epsilon)=\inf_{\lambda', \, \epsilon'} \{ N(\lambda', \epsilon') \mid \lambda' \geq \lambda, \, \epsilon' \geq \epsilon \}+1
\]
We see that by construction $N'$ is non-decreasing.

To arrange that $N'$ is continuous in the second coordinate, we first note that since we have arranged that $N'$ is non-decreasing, then if we fix a $\lambda \in \mathbb{R}_{\geq 1}$, the function $N'(\lambda, \epsilon) \colon \mathbb{R}_{\geq 0} \to \mathbb{R}_{\geq 0}$ is non-decreasing and thus (Riemann) integrable. We then set 
\[
\hat{N}(\lambda, \epsilon) = \int_{\epsilon}^{\epsilon+1} N'(\lambda, t) \, dt.
\]
A standard calculation using that $\hat N(\lambda, \cdot)$ is bounded on compact intervals shows that $\widehat{N}$ is continuous in the second coordinate. We also note that, since $N'$ is non-decreasing, we have $N' \leq \widehat{N}$.

Let $\alpha$ be an $N$-Morse geodesic and let $\sigma$ be a $(\lambda, \epsilon)$-geodesic with endpoints on $\alpha$. We note that $\sigma$ is also an $N(\lambda', \epsilon')$-quasi geodesic for any $\lambda' \geq \lambda$ and $\epsilon' \geq \epsilon$. Thus by construction, we know that $\sigma$ must be in the $N'(\lambda, \epsilon)$-neighborhood of $\alpha$ and since $N' \leq \widehat{N}$, we conclude that $\sigma$ is in the $\widehat{N}(\lambda, \epsilon)$-neighborhood of $\alpha$.
\end{proof}

We now state and prove the corrected version of Lemma \ref{M-Morse uniformly converge to M-Morse}.

\begin{lem}[Lemma 2.10 redux]
Let $X$ be a geodesic metric space and suppose that $\{ \gamma_i \colon \mathbb{R}_{\geq 0} \to X\}$ is a sequence of refined $N$-Morse geodesic rays that converge uniformly on compact sets to a geodesic ray $\gamma$. Then $\gamma$ is refined $N$-Morse. 
\end{lem}

\begin{proof}
Let $\eta > 0$. Let $\sigma$ be a $(\lambda, \epsilon)$-quasi-geodesic with endpoints $\gamma(s)$ and $\gamma(t)$ on $\gamma$. Since the $\gamma_i$ converge uniformly on compact sets and are refined $N$-Morse, there exists an $I \in \mathbb{N}$ such that for any $i \geq I$, we have $d(\gamma_i(t), \gamma(t)) \leq \eta$ on $\gamma|_{[s,t]}$. It follows that $\gamma$ is $N'$-Morse where $N'(\lambda, \epsilon) = N(\lambda, \epsilon +\eta)+ \eta$. Since this is true for all $\eta >0$ and $N$ is continuous in the second coordinate, we can conclude that $\sigma$ is in the closed $N(\lambda, \epsilon)$ neighborhood of $\gamma|_{[s,t]}$. It follows that $\gamma$ is refined $N$-Morse. 
\end{proof}

We wish to reprove Corollary \ref{maps continuous} to avoid using Lemma \ref{M-Morse uniformly converge to M-Morse} from the same paper so that we may define the Morse boundary without reference to refined Morse geodesics.

\begin{lem} \label{lem:N-Morse close}
	Let $X$ be a geodesic metric space and let $\alpha\colon [0, A] \to X$ be an $N$-Morse geodesic and $\beta \colon [0,A] \to X$ be any geodesic. Also assume that $\alpha(0)=\beta(0)$. Then for $t\in [0, A-d(\alpha(A),\beta(A))]$ we have $d(\alpha(t),\beta(t))<4N(3,0)$.
\end{lem}

\begin{proof}
	Let $\alpha(x)$ be the closest point on $\alpha$ to $\beta(b)$. By the triangle inequality $x \geq A-d(\alpha(A),\beta(A))$. It follows as in the proof of CASE 1 of Proposition \ref{morse close}, that the concatenation $\sigma =\beta([0,A]) \cup [\beta(A), \alpha(x)]$ is a $(3,0)$-quasi-geodesic. Since $\alpha$ is $N$-Morse, we know that $\sigma$ is in the $N(3,0)$-neighborhood of $\alpha([0,y])$ and by Lemma \ref{near morse means hausdorff close} we can bound the Hausdorff distance between $\alpha([0,y])$ and $\sigma$ by $2N(3,0)$. By a standard argument we may conclude that for all $t \in [0, A-d(\alpha(A),\beta(A))]$, we have $d(\alpha(t),\beta(t))<4N(3,0)$.
\end{proof}

We now reprove Corollary \ref{maps continuous}:

\begin{lem}[Corollary 3.2 redux] \label{lem:inclusion continuous}
Let $N$ and $N'$ be Morse gauges such that every $N$-Morse geodesic is also $N'$-Morse. Then the natural inclusion $i \colon \partial^N_M X_p \to \partial^{N'}_M X_p$ is continuous. In particular if $N \leq N'$ then this condition is satisfied.
\end{lem}

\begin{proof}
Let $U$ be an open set in $\partial^{N'}_M X_p$. We wish to show that $i^{-1}(U)$ is open. Let $x \in i^{-1}(U)$ and $\alpha_x$ a geodesic ray representing $x$. Since $i$ is an inclusion and $U$ is open in $\partial^{N'}_M X_p$, then there exists a $j \in \mathbb{N}$ so that $V_j^{N'}(\alpha_x) \subset U$. Let $k=j + 3\delta_{N}$. Applying Lemma \ref{lem:N-Morse close} thinking of $\alpha_x$ as an $N'$-Morse geodesic, we know that for any $y \in V_k^{N}(\alpha_x)$ and any geodesic $\alpha_y$ representing $y$, that $d(\alpha_x(t),\alpha_y(t))<4N'(3,0)$ for $t \in [0,j]$. So $i(V_k^{N}(\alpha_x))\subset V_j^{N'}(\alpha_x)\subset U$. Since we can do this for any $x \in i^{-1}(U)$ we can conclude that $i^{-1}(U)$ is open. 
\end{proof}

\begin{rem} \label{rem:deltan}
We note that one can also prove this using Corollary \ref{morse close same basepoint}, but we prove Lemma \ref{lem:N-Morse close} because it is a useful lemma of independent interest. In fact, in future research on the Morse boundary one might wish to redefine $\delta_N$ to be $4N(3,0)$. 
\end{rem}



\begin{defn}
The refined Morse boundary, denoted $\partial_{\widehat{M}} X_p$, is defined the same way as the Morse boundary as in Section 
\ref{sec:morse boundary defn}, but rather than considering all Morse gauges $\mathcal{M}$, we consider only refined Morse gauges $\widehat{\mathcal{M}}$.
\end{defn}

\begin{thm}\label{thm:homeo}
Let $X$ be a proper geodesic metric space. Then  the inclusion $\partial_{\widehat{M}} X \hookrightarrow \partial_M X$ induces a homeomorphism. 
\end{thm}

\begin{proof}
This follows from Lemma \ref{lem:every gauge gives refined}, Lemma \ref{lem:inclusion continuous}, and the universal property of direct limits. 
\end{proof}

\bibliography{library.bib}
\bibliographystyle{amsalpha}
\end{document}